\documentclass{amsart}
\usepackage[all]{xy}
\usepackage[usenames,dvipsnames]{xcolor}
\usepackage{amsthm}
\usepackage{amssymb}
\usepackage[colorlinks=true]{hyperref}
\usepackage{tikz}
\usetikzlibrary{matrix,arrows}

\numberwithin{equation}{section}

\newtheorem{theorem}[equation]{Theorem} 
\newtheorem*{theorem*}{Theorem}
\newtheorem{lemma}[equation]{Lemma}
\newtheorem{proposition}[equation]{Proposition}
\newtheorem{corollary}[equation]{Corollary}
\newtheorem*{corollary*}{Corollary}
\newtheorem{conjecture}[equation]{Conjecture}

\theoremstyle{remark}
\newtheorem{definition}[equation]{Definition}

\newtheorem{example}[equation]{Example}

\newtheorem{notation}[equation]{Notation}

\theoremstyle{remark}
\newtheorem{remark}[equation]{Remark}

\newcommand{\cA}{{\mathcal A}}
\newcommand{\cB}{{\mathcal B}}
\newcommand{\cC}{{\mathcal C}}
\newcommand{\cD}{{\mathcal D}}
\newcommand{\cE}{{\mathcal E}}
\newcommand{\cF}{{\mathcal F}}

\newcommand{\cM}{{\mathcal M}}

\newcommand{\cO}{{\mathcal O}}

\newcommand{\cT}{{\mathcal T}}

\newcommand{\ch}{{\mathrm{ch}}}
\newcommand{\Td}{{\mathrm{Td}}}

\newcommand{\bbC}{\mathbb{C}}

\newcommand{\bbP}{\mathbb{P}}

\newcommand{\bbQ}{\mathbb{Q}}
\newcommand{\bbZ}{\mathbb{Z}}

\DeclareMathOperator{\SmProj}{SmProj} 

\DeclareMathOperator{\Id}{Id}
\DeclareMathOperator{\id}{id}

\DeclareMathOperator{\NChow}{NChow} 
\DeclareMathOperator{\Chow}{Chow} 
\DeclareMathOperator{\Num}{Num} 

\newcommand{\Ab}{\mathrm{Ab}}

\newcommand{\dgcat}{\mathsf{dgcat}}

\newcommand{\perf}{\mathrm{perf}}

\newcommand{\dg}{\mathsf{dg}}

\newcommand{\Hom}{\mathrm{Hom}}

\newcommand{\rep}{\mathrm{rep}}

\newcommand{\Var}{\mathsf{Var}}
\newcommand{\Cat}{\mathsf{Cat}}

\newcommand{\dgHo}{\mathsf{H}^0}

\newcommand{\Hmo}{\mathsf{Hmo}}
\newcommand{\op}{\mathsf{op}}

\newcommand{\too}{\longrightarrow}

\newcommand{\ie}{\textsl{i.e.}\ }
\newcommand{\eg}{\textsl{e.g.}}

\begin{document}

\title[Polarized intermediate Jacobians via NC motives]{From semi-orthogonal decompositions\\to polarized intermediate Jacobians\\via Jacobians of noncommutative motives}

\author{Marcello Bernardara and Gon{\c c}alo~Tabuada}

\address{Institut de Math\'ematiques de Toulouse \\ %
Universit\'e Paul Sabatier \\ %
118 route de Narbonne \\ %
31062 Toulouse Cedex 9\\ %
France}
\email{marcello.bernardara@math.univ-toulouse.fr} 
\urladdr{http://www.math.univ-toulouse.fr/~mbernard/}

\address{Gon{\c c}alo Tabuada, Department of Mathematics, MIT, Cambridge, MA 02139, USA}
\email{tabuada@math.mit.edu}
\urladdr{http://math.mit.edu/~tabuada/}
\thanks{G.~Tabuada was partially supported by a NSF CAREER Award}
\date{\today}


\abstract{Let $X$ and $Y$ be complex smooth projective varieties, and $\cD^b(X)$ and $\cD^b(Y)$ the associated
bounded derived categories of coherent sheaves. Assume the existence of a triangulated category $\cT$ which is admissible both
in $\cD^b(X)$ as in $\cD^b(Y)$. Making use of the recent theory of Jacobians of noncommutative motives, we construct out of this categorical
data a morphism $\tau$ of abelian varieties (up to isogeny) from the product of the intermediate algebraic Jacobians of $X$
to the product of the intermediate algebraic Jacobians of $Y$. When the orthogonal complement $\cT^\perp$ of $\cT\subset\cD^b(X)$
has a trivial Jacobian (\eg\ when $\cT^\perp$ is generated by exceptional objects), the morphism $\tau$ is split injective. 
When this also holds for the orthogonal complement $\cT^\perp$ of $\cT \subset \cD^b(Y)$, $\tau$ becomes an isomorphism. Furthermore, in the case where
$X$ and $Y$ have a single intermediate algebraic Jacobian carrying a principal polarization, we prove that $\tau$ preserves
this extra piece of structure.}

As an application, we obtain categorical Torelli theorems, an incompatibility between two conjectures
of Kuznetsov (one concerning Fourier-Mukai functors and another one concerning Fano threefolds), a new proof of a classical theorem
of Clemens-Griffiths on blow-ups of threefolds, and several new results on quadric fibrations and intersection of quadrics.}
\maketitle 
\vskip-\baselineskip
\vskip-\baselineskip
\vskip-\baselineskip
\section{Introduction}\label{sec:introduction}
At their ICM address \cite{bondal_orlov:ICM2002}, Bondal-Orlov suggested that classical birational geometry could be
studied using bounded derived categories of coherent sheaves and their semi-orthogonal decompositions. This motivates the general questions:

\vspace{0.1cm}

{\bf Question:} {\it Does the bounded derived category $\cD^b(X)$ of a complex smooth projective variety $X$ carries information
about the intermediate Jacobians and their polarizations~? If so, how can they be ``extracted'' from $\cD^b(X)$ ?}

\vspace{0.1cm}

In the case of cubic threefolds, conic bundles, and more generally rationally representable Fano threefolds, some partial
answers are known; see
\cite{berna_macri_mehro_stella,bolognesi_bernardara:conic_bundles,bolognesi_bernardara:representability}.
The goal of this article is to show that if one replaces
$\cD^b(X)$ by its unique differential graded enhancement, then the above general questions admit precise affirmative~answers.

\subsection*{Polarized intermediate Jacobians}
Let $X$ be an irreducible smooth projective $\bbC$-scheme of dimension $d_X$. Recall from Griffiths \cite{Griffiths} the construction
of the Jacobians $J^i(X), 0 \leq i \leq d_X-1$. In contrast with the Picard $J^0(X)= \mathrm{Pic}^0(X)$ and the Albanese
$J^{d_X-1}(X)=\mathrm{Alb}(X)$ varieties, the intermediate Jacobians are {\em not} algebraic.  Nevertheless, they contain an algebraic
torus $J^i_a(X)\subseteq J^i(X)$ defined by the image of the Abel-Jacobi map
\begin{eqnarray}\label{eq:AbelJacobi}
AJ^i: A^{i+1}_\bbZ(X) \to J^i(X) && 0 \leq i \leq d_X-1\,,
\end{eqnarray}
where $A^{i+1}_\bbZ(X)$ stands for the group of algebraically trivial cycles of codimension $i+1$; consult Vial \cite[\S2.3]{Vial}
for further details. The map \eqref{eq:AbelJacobi} is surjective when $i=0$ and $i=d_X-1$ and so $J^0_a(X)=\mathrm{Pic}^0(X)$ and
$J^{d_X-1}_a(X)=\mathrm{Alb}(X)$. 

In general the abelian varieties $J^i_a(X)$ are only well-defined up to isogeny. However, in the case of curves, Fano threefolds,
even dimensional quadric fibrations over $\bbP^1$, odd dimensional quadric fibrations over rational surfaces, and also in the case of the intersection
of two (resp. three) quadrics of odd (resp. even) dimension, there is a single non-trivial algebraic Jacobian $J(X):=J_a^{(d_X-1)/2}(X)$ which carries moreover a canonical {\em principal polarization}; see Clemens-Griffiths \cite{clemensgriffiths}. This extra piece of structure is of
major importance. For instance, in the case of Fano threefolds $X$, the abelian variety $J(X)$ endowed with its canonical principal polarization contains
all the information about the birational class of $X$.
\subsection*{More detailed questions}
Let $X$ and $Y$ be two irreducible smooth projective $\bbC$-schemes of dimensions $d_X$ and $d_Y$, respectively, and $\cD^b(X)$ and $\cD^b(Y)$ the associated bounded derived categories of coherent sheaves; consult Huybrechts \cite[\S3]{huybrechts}. Assume that $X$ and $Y$ are related by the following {\bf categorical data:}

\vspace{0.2cm}

{\it There exist semi-orthogonal decompositions $\cD^b(X) = \langle \cT_X, \cT_X^{\perp}\rangle$ and $\cD^b(Y) = \langle\cT_Y, \cT_Y^{\perp}\rangle
$ (see \cite[\S4]{huybrechts}) and an equivalence $\phi: \cT_X \stackrel{\simeq}{\to} \cT_Y$ of triangulated categories}. 

\vspace{0.2cm}

Here is a list of examples of such categorical data:
\begin{itemize}
\item[(E0)] $X$ arbitrary, $\cT_X:= \cD^b(X)$, and $Y$ such that $\cT_Y:=\cD^b(Y)$ is equivalent to $\cD^b(X)$. This clearly holds when $X$ is isomorphic to $Y$. Thanks to the work of Bondal-Orlov, Bridgeland, and Mukai (see \cite{bondal-orlov,bridge-flop,mukai} and also \cite[\S11.4]{huybrechts}), this holds also when $X$ and $Y$ are two crepant resolutions of a threefold $Z$ with terminal singularities (\eg\ $X$ and $Y$ two birational Calabi-Yau threefolds), when $X$ and $Y$ are related by a Mukai flop, and also when $X$ is an abelian variety and $Y$ is its dual $\widehat{X}$.
\item[(E1)] $X$ arbitrary, $\cT_X:=\cD^b(X)$, and $Y$ obtained from $X$ by a standard flip or by blow-up along a smooth irreducible subscheme $Z \subseteq X$; see Orlov \cite{orlovprojbund}.
\item[(E2)] $X$ a hyperelliptic curve, $\cT_X:= \cD^b(X)$, and $Y$ a complete intersection of two even dimensional quadrics; see Bondal-Orlov \cite{bondal-orlov}.
\item[(E3)] $X$ a hyperelliptic or trigonal curve, $\cT_X:=\cD^b(X)$, and $Y$ a rational conic bundle over a Hirzebruch
surface or a rational del Pezzo fibration of degree $4$ over $\bbP^1$; see \cite{auel-bernar-bologn,bolognesi_bernardara:conic_bundles}.
In the same vein, $X\subset \bbP^3$ a smooth curve of genus $5$ and degree $7$ (or a smooth curve of genus $2$), $\cT_X:= \cD^b(X)$, and $Y$ a rational conic bundle over $\bbP^2$; see \cite{bolognesi_bernardara:conic_bundles}.
\item[(E4)] $X$ (resp. $Y$) a Fano threefold of index $1$ (resp. $2$), and $\cT_X$ (resp. $\cT_Y$) the orthogonal complement of an exceptional collection of objects in $\cD^b(X)$ (resp. in $\cD^b(Y)$); see Kuznetsov \cite{kuznetfanothreefolds}.
\item[(E5)] $X$ the intersection of a family of quadrics of Fano (or Calabi-Yau) type, $\cT_X$ the orthogonal complement of an exceptional collection of objects in $\cD^b(X)$, and $Y$ the fibration generated by the family of quadrics; see Kuznetsov \cite{kuznetquadrics}.
\item[(E6)] $X$ a quadric fibration (over a smooth projective $\bbC$-scheme $S$) endowed with a regular section, $Y$ the $\bbC$-scheme obtained from $X$ by hyperbolic reduction, and $\cT_X$ and $\cT_Y$ the derived categories of perfect complexes over the sheaves of even parts of Clifford algebras on $S$; see \cite{auel-bernar-bologn}.
\item[(E7)] $X$ arbitrary, $\cT_X:=\cD^b(X)$, and $Y \to X$ a flat fibration for which $Y$ can be embedded in a projective bundle $\bbP(E) \to X$ such that $\omega_{Y/X} = \cO_{\bbP(E)/X}(-l)_{\vert X}$ with $l >0$; see \cite{auel-bernar-bologn}. This example is inspired by Orlov's pioneering work \cite{orlovprojbund}.
\item[(E8)] $X$ an elliptic curve (resp. a curve of degree $42$), $\cT_X = \cD^b(X)$, and $Y$ a 5-dimensional linear section of
the Grassmannian $\mathrm{Gr}(2,6)$ (resp. $\mathrm{Gr}(2,7)$). We can also replace $42$ by $14$ and $\mathrm{Gr}(2,7)$
by the Pfaffian $\mathrm{Pf}(4,7)$; see 
Kuznetsov~\cite{kuznet-grass}.
\item[(E?)] {\em Homological projective duality} (see Kuznetsov \cite{kuznetsov:hpd}) is expected to provide a very large class of other examples.
\end{itemize}
In Examples (E2)-(E4) and (E5) (in the particular case of the intersection of two (resp. three) quadrics of even (resp. odd) dimension),
there is a single non-trivial algebraic Jacobian $J(X)$ (and $J(Y)$) which carries moreover a canonical principal polarization.
This holds also in same particular cases of the remaining examples (\eg\ even (resp. odd) dimensional quadric fibrations over
$\bbP^1$ (resp. over rational surfaces)); an exhaustive list can be found in \cite[Remark 3.8]{bolognesi_bernardara:representability}.
It is now natural to ask the following precise questions:

\vspace{0.1cm}

{\bf Question A:} {\it Does the above categorical data gives rise to a well-defined morphism $\tau$ of abelian varieties (up to isogeny) from $\prod_{i=0}^{d_X-1} J_a^i(X)$ to $\prod_{i=0}^{d_Y-1} J_a^i(Y)$~?}

\vspace{0.1cm}

{\bf Question B:} {\it In the above Examples (E2)-(E4) (and in all the other particular cases), does the morphism $\tau$ preserves the principal polarization ?}

\vspace{0.1cm}

In this article, making use of the recent theory of Jacobians of noncommutative motives \cite{MT}, we provide precise affirmative answers to these questions. Consult \S\ref{sec:app1}-\ref{sec:app-4} for several applications.
\subsection*{Jacobians of noncommutative motives}
Recall from Andr{\'e} \cite[\S4]{Andre} the construction of the (contravariant) functor $M_\bbQ(-):\SmProj(\bbC)^\op \to \Chow(\bbC)_\bbQ$
from smooth projective $\bbC$-schemes to Chow motives. De Rham cohomology $H^\ast_{dR}(-)$ factors through $\Chow(k)_\bbQ$. Hence, given an irreducible smooth projective $\bbC$-scheme $X$ of dimension $d_X$, one defines
\begin{eqnarray}\label{eq:NH}
NH_{dR}^{2i+1}(X):= \sum_{C,\gamma_i} \mathrm{Im} \big(H^1_{dR}(C) \stackrel{H^1_{dR}(\gamma_i)}{\too} H^{2i+1}_{dR}(X) \big) && 0 \leq i \leq d_X-1\,,
\end{eqnarray}
where $C$ is a smooth projective curve and $\gamma_i: M_\bbQ(C) \to M_\bbQ(X)(i)$ a morphism in $\Chow(\bbC)_\bbQ$. Intuitively speaking, \eqref{eq:NH} are the odd pieces of de Rham cohomology that are generated by curves. By restricting the classical intersection bilinear pairings on de Rham cohomology (see \cite[\S3.3]{Andre}) to these pieces one then obtains
\begin{eqnarray}\label{eq:pairings1}
\langle-,- \rangle : NH_{dR}^{2d_X-2i-1}(X) \times NH_{dR}^{2i+1}(X) \too \bbC && 0 \leq i \leq d_X-1\,.
\end{eqnarray}
Recall from \S\ref{sec:NCChow} the construction of the category $\NChow(\bbC)_\bbQ$ of noncommutative Chow motives. 
Examples of noncommutative Chow motives include finite dimensional $\bbC$-algebras of finite global dimension as well 
as the unique dg enhancements $\perf^\dg(X)$ (see Lunts-Orlov \cite{LO}) of the derived categories $\perf(X)$ of perfect 
complexes. Note that since $X$ is smooth, every complex of coherent sheaves is perfect (up to quasi-isomorphism). Hence, the canonical inclusion $\perf(X) \hookrightarrow \cD^b(X)$ is an equivalence of categories. Now, recall from \cite{MT} 
the construction of the {\em Jacobian} functor
$$ {\bf J}(-): \NChow(\bbC)_\bbQ \too \Ab(\bbC)_\bbQ$$
with values in the category of abelian $\bbC$-varieties up to isogeny. Among other properties, one 
has an isomorphism ${\bf J}(\perf^\dg(X)) \simeq \prod_{i=0}^{d_X-1} J_a^i(X)$ whenever the above pairings \eqref{eq:pairings1} 
are non-degenerate. As explained in {\em loc. cit.}, this is always the case for $i=0$ and $i=d_X-1$ and the remaining cases 
follow from Grothendieck's standard conjecture of Lefschetz type. Hence, the pairings \eqref{eq:pairings1} are non-degenerate 
for curves, surfaces, abelian varieties, complete intersections, uniruled threefolds, rationally connected fourfolds, 
and for any smooth hypersurface section, product, or finite quotient thereof (and if one trusts Grothendieck 
they are non-degenerate for all smooth projective $\bbC$-schemes).
\section{Statement of results}
Let $X$ and $Y$ be two irreducible smooth projective $\bbC$-schemes of dimensions $d_X$ and $d_Y$, respectively. Assume that $X$ and $Y$ are related by the above categorical data. In what follows, we will write $\Phi$ for the composition $ \perf(X) \to \cT_X \stackrel{\phi}{\simeq} \cT_Y \hookrightarrow \perf(Y)$, where the first functor is the projection.
\begin{conjecture}{(Kuznetsov \cite[Conjecture 3.7]{kuznetsov:hpd})}\label{conj:Kuznetsov1}
The functor $\Phi$ is of {\em Fourier-Mukai type}, \ie there exists a perfect complex $\cE \in \perf(X\times Y)$ such that $\Phi$ is isomorphic to the Fourier-Mukai functor $\Phi_\cE(-):= Rq_\ast(p^\ast(-)\otimes^L \cE)$, where $p:X\times Y \to X$ and $q:X\times Y \to Y$ are the projection morphisms. 
\end{conjecture}
Thanks to the work of Bondal-Orlov, Kuznetsov, Orlov, and others
(see \cite{auel-bernar-bologn, bolognesi_bernardara:conic_bundles,bondal-orlov,kuznetsov:hpd,kuznetquadrics,orlovprojbund}),
Conjecture \ref{conj:Kuznetsov1} is known to be true in Examples (E0)-(E3), (E5)-(E6), (E8) and (E?).
As mentioned above, $\perf(X)$ and $\perf(Y)$ admit (unique) dg enhancements $\perf^\dg(X)$ and $\perf^\dg(Y)$. Let us then denote by
$\cT_X^\dg, \cT_X^{\perp,\dg}, \cT_Y^\dg, \cT_Y^{\perp,\dg}$ the inherited dg enhancements. Making use of them, we prove
that Kuznetsov's Conjecture~\ref{conj:Kuznetsov1} is equivalent to the existence of a dg enhancement $\Phi^\dg:\perf^\dg(X) \to \perf^\dg(Y)$ of $\Phi$;
see Proposition~\ref{prop:FMequalsdg}. This result is known to experts but, to the best
of authors' knowledge, no proof can be found in the litterature. Our first main result, which answers affirmatively to Question A, is the following:
\begin{theorem}\label{thm:main1}
Let $X$ and $Y$ be two irreducible smooth projective $\bbC$-schemes of dimensions $d_X$ and $d_Y$, respectively, that are related by the above categorical data. 
\begin{itemize}
\item [(i)] Assume that the bilinear pairings \eqref{eq:pairings1} (associated to $X$ and $Y$) are non-degenerate, and that Conjecture \ref{conj:Kuznetsov1} holds. Under these assumptions, one obtains a well-defined morphism
$\tau: \prod_{i=0}^{d_X-1} J^i_a(X) \to \prod_{i=0}^{d_Y-1} J^i_a(Y)$ in $\Ab(\bbC)_\bbQ$. 

\item[(ii)] Assume moreover that ${\bf J}(\cT_X^{\perp,\dg})=0$. This holds for instance whenever $\cT_X^\perp$ admits a full exceptional collection. Under this extra assumption, the morphism $\tau$ is split injective.

\item[(iii)] Assume furthermore that ${\bf J}(\cT_Y^{\perp,\dg})=0$. Under this extra assumption, the morphism $\tau$ becomes an isomorphism.
\end{itemize}
\end{theorem}
As mentioned above, the bilinear pairings \eqref{eq:pairings1} are non-degenerate whenever the Grothendieck standard conjecture of Lefschetz
type holds. This is the case in Examples (E2)-(E5), (E6) whenever $\mathrm{dim}(S)\leq 2$, and in some particular cases of (E7); see Vial \cite[\S 7]{vial-fibrations}. In what concerns Examples (E0)-(E1), it suffices that the standard conjecture of Lefschetz type holds for $X$.
The extra assumption ${\bf J}(\cT_X^{\perp,\dg})=0$ holds in Examples (E0)-(E3) and (E7)-(E8) since $\cT_X^\perp=0$ and in Examples (E4)-(E5) since $\cT_X^\perp$ admits a full exceptional collection. In what concerns Example (E6), Kuznetsov \cite{kuznetquadrics} proved that $\cT_X^\perp$ admits a semi-orthogonal decomposition where each piece is a copy of $\perf(S)$. Hence, ${\bf J}(\cT_X^{\perp,\dg})=0$ whenever ${\bf J}(\perf^\dg(S))=0$.

\medskip

Now, in order to address Question B, we introduce the following notion:
\begin{definition}{(see Definition \ref{def:strong-repre})}\label{def:veryreprenstable-intro}
An irreducible smooth projective $\bbC$-scheme $X$ of odd dimension $d_X=2n+1$ is called {\em verepresentable}\footnote{The fusion of the words ``very'' and ``representable''.} if:
\begin{itemize}
\item[(i)] the group of algebraically trivial cycles $A_\bbZ^{i+1}(X)$ is trivial for $i\neq n$;
\item[(ii)] the group $A_\bbZ^{n+1}(X)$ admits an {\em algebraic representative} carrying an {\em incidence polarization};
 see \S\ref{sub:principal}. Intuitively speaking, $A^{n+1}_{\bbZ}(X)$ is identified in a universal way with a principally polarized abelian variety.
\item[(iii)] the Abel-Jacobi map $AJ^n(X) : A^{n+1}_{\bbZ}(X) \twoheadrightarrow J_a^n(X)$ gives rise to an isomorphism $A_\bbQ^{n+1}(X) \simeq J_a^n(X)_\bbQ$.
\end{itemize}
\end{definition}

As explained in Examples \ref{ex:very1}-\ref{ex:very3}, the $\bbC$-schemes $X$ and $Y$ of Examples (E2)-(E4) and (E5)
(in the particular case of the intersection of two (resp. three) quadrics of even (resp. odd) dimension) are verepresentable. By combining the above definition \eqref{eq:AbelJacobi} of $J_a^i(X)$ with item (i) of Definition~\ref{def:veryreprenstable-intro}
one observes that whenever $X$ is verepresentable, $J_a^i(X)=0$ for $i \neq n$.
Consequently, there is a single non-trivial algebraic Jacobian $J(X):=J^n_a(X)$ which, thanks to item (ii)
of Definition~\ref{def:veryreprenstable-intro}, carries a canonical principal polarization. Moreover, item (iii) of Definition~\ref{def:veryreprenstable-intro} shows that this principally polarized abelian variety is isomorphic, up to isogeny, to $A^{n+1}_{\bbZ}(X)$. Our second main result, which answers affirmatively to Question B, is the following:
\begin{theorem}\label{thm:main2}
Let $X$ and $Y$ be two irreducible smooth projective $\bbC$-schemes as in items (i)-(ii) of Theorem~\ref{thm:main1}. Assume moreover that $X$ and $Y$ are verepresentable
and that the canonical bundle of $X$ is ample, antiample, or trivial.
Under these assumptions, the split injective morphism $\tau: J(X) \to J(Y)$ preserves the principal polarization. When ${\bf J}(\cT_Y^{\perp,\dg})=0$ the morphism $\tau$ becomes an isomorphism.
\end{theorem}
Note that the canonical bundle of $X$ is always ample, antiample, or trivial in Examples (E2)-(E5), and (E8).
The extra assumption ${\bf J}(\cT_Y^{\perp,\dg})=0$ holds in Example (E0)
since $\cT_Y^\perp=0$ and in Examples (E2)-(E5) and (E8) since $\cT_Y^\perp$ admits a full exceptional collection. 
In conclusion, Examples (E2)-(E3) and (E5) (in the particular case of the intersection of two (resp. three) quadrics of even
(resp. odd) dimension) satisfy all the assumptions of Theorem~\ref{thm:main2}. In what concerns Example (E4),
this is also the case if one assumes Kuznetsov's conjecture \ref{conj:Kuznetsov1}. 
In what concerns Example (E8), all the assumptions of Theorem \ref{thm:main2} are satisfied, except the verepresentability
of $Y$ and the non-degeneracy of the pairings \eqref{eq:pairings1} for $Y$; see Conjecture \ref{conj:grass-pfaff}.
\begin{remark}
Due to their generality, functoriality, and simplicity, Theorems \ref{thm:main1} and \ref{thm:main2}
should provide a useful toolkit of every mathematician working with derived categories of schemes.
For example, Theorem \ref{thm:main1} has recently played a key role in the proof of new cases of a conjecture of
Paranjape-Srinivas on Chow groups of intersections of quadrics; see \cite{marcello-goncalo-chowgroups}.
\end{remark}

\medbreak\noindent\textbf{Notations:} Throughout the article we will work always over the field $\bbC$ of complex numbers.
All $\bbC$-schemes will be assumed to be smooth, proper, and irreducible.

\medbreak\noindent\textbf{Acknowledgments:} The authors are very grateful to Thomas Dedieu, Valery Lunts, Alexander Polishchuk, Paolo Stellari, Michel Vaqui\'e and Charles Vial for useful discussions, e-mail exchanges, and motivating questions. G.~Tabuada would like also to thank the MSRI for its hospitality and excellent working conditions.
\section{Application I: Categorical Torelli}\label{sec:app1}
Let $X$ and $Y$ be two verepresentable $\bbC$-schemes. The (generalized) {\em Torelli theorem} claims that $X\simeq Y$
if and only if $J(X)\simeq J(Y)$ as principally polarized abelian varieties. The particular case of curves (\ie the classical Torelli theorem) was proved by Torelli \cite{torelli} one hundred years ago. 
Thanks to the work of Clemens-Griffiths, Debarre, Donagi, Laszlo, M\'erindol, and Voisin
(see \cite{clemensgriffiths,debarre-qci,debarre-quartic-double,donagi-torelli,laszlo,merindol,voisin-quartic-double}), this holds also in the case of cubic threefolds, quartic double solids, and intersections of two (resp. three) quadrics of even (resp. odd) dimension. Theorem~\ref{thm:main2} furnishes us automatically the following categorification: 

\begin{corollary}{(Categorical Torelli)}\label{cor:Torelli}
Let $X$ and $Y$ be two $\bbC$-schemes as in Theorem~\ref{thm:main2} with ${\bf J}(\cT^{\perp,\dg}_Y)=0$.
Whenever the (generalized) Torelli theorem holds, the $\bbC$-schemes $X$ and $Y$ are isomorphic.
\end{corollary}
We now illustrate Corollary~\ref{cor:Torelli} in four different situations:
\begin{itemize}
\item[(i)] {\bf Intersections of two even dimensional quadrics:} let $X$ and $Y$ be intersections of two even dimensional quadrics, and $C_X$ and $C_Y$ the associated
hyperelliptic curves; see Reid \cite[Thm.~1.10]{reid:thesis}. As  proved by Bondal-Orlov~\cite{bondal-orlov}, one has fully faithful functors
$\cT_X:=\perf(C_X) \to \perf(X)$ and $\cT_Y:=\perf(C_Y) \to \perf(Y)$ whose orthogonal complements are generated
by exceptional objects. Moreover, $\perf(C_X) \simeq \perf(\bbP^1,\cC_{0,X})$ (resp. 
$\perf(C_Y) \simeq \perf(\bbP^1,\cC_{0,Y})$), where $\cC_{0,X}$ (resp. $\cC_{0,Y}$) is the sheaf of the even
parts of the Clifford algebra associated to the quadric span $Q_X \to \bbP^1$ (resp. $Q_Y \to \bbP^1$); see Kuznetsov \cite[Corollary 5.7]{kuznetquadrics}. Since $X$ and $Y$ are complete intersections, the pairings
\eqref{eq:pairings1} are non-degenerate. Moreover, $X$ and $Y$ are verepresentable (see \cite[Thm. 1.5]{marcello-goncalo-chowgroups} for item (i),
Reid \cite{reid:thesis} for item (ii) and Donagi \cite[Thm.~3.3]{donagi-torelli} for item (iii))
and the canonical bundle of $X$ is either trivial (when $X$ is an elliptic curve and $C_X \simeq X)$ or antiample. Now, suppose that there exists an equivalence $\phi: \perf(C_X) \simeq \perf(C_Y)$ of triangulated categories. Under such assumptions, the composed functor $\Phi:\perf(X) \to \perf(Y)$ is of Fourier-Mukai type since this is the case of the projection $\perf(X) \to \perf(C_X)$ (see \cite[Thm.~7.1]{kuznetbasechange}), the equivalence $\phi:\perf(C_X) \simeq \perf(C_Y)$ (see Orlov \cite{orlov-represent}), and also of the inclusion $\perf(C_Y) \hookrightarrow \perf(Y)$ (see Bondal-Orlov  \cite{bondal-orlov}). As a consequence, all the assumptions of Theorem~\ref{thm:main2} are satisfied.
The (generalized) Torelli theorem holds in this case; see Donagi \cite{donagi-torelli}. Hence, we obtain from Corollary~\ref{cor:Torelli} the following implication:
\begin{equation*}
\perf(C_X) \simeq \perf(C_Y) \Rightarrow X \simeq Y\,.
\end{equation*}
\item[(ii)] {\bf Intersections of three odd dimensional quadrics:} let $X$ and $Y$ be intersections of three odd dimensional quadrics
of Fano type, and $\cC_{0,X}$ (resp. $\cC_{0,Y}$) the sheaf of the even
parts of the Clifford algebra associated to the quadric span $Q_X \to \bbP^2$
(resp. $Q_Y \to \bbP^2$). Similarly to item (i), the triangulated category $\cT_X:= \perf(\bbP^2,\cC_{0,X})$
(resp. $\cT_Y:= \perf(\bbP^2,\cC_{0,Y})$) embeds in $\perf(X)$ (resp. in $\perf(Y)$) and its orthogonal
complement is generated by exceptional objects. Since $X$ and $Y$ are complete intersections the bilinear pairings
\eqref{eq:pairings1} are non-degenarate. Moreover, $X$ and $Y$ are verepresentable (see \cite[Thm. 1.5]{marcello-goncalo-chowgroups} for item (i)
and Beauville \cite[Thm.~6.3]{beauvilleprym} for items (ii)-(iii))
and the canonical bundle of $X$ is antiample. Now, suppose that there exists an equivalence $\phi: \perf(\bbP^2,\cC_{0,X}) \simeq 
\perf(\bbP^2,\cC_{0,Y})$ of triangulated categories. In contrast with item (i), it is not known if the composed functor 
$\Phi:\perf(X) \to \perf(Y)$ is of Fourier-Mukai type; only the projection $\perf(X) \to \perf(\bbP^2,\cC_{0,X})$ and the
embedding $\perf(\bbP^2,\cC_{0,Y}) \hookrightarrow \perf(Y)$ are known to be of Fourier-Mukai type; see
\cite[Thm.~7.1]{kuznetbasechange}\cite{kuznetquadrics}. As a consequence, if one assumes that 
$\perf(\bbP^2,\cC_{0,X})$ and $\perf(\bbP^2,\cC_{0,Y})$ are Fourier-Mukai equivalent, we obtain from
Corollary~\ref{cor:Torelli} the following implication
\begin{equation*}
\perf(\bbP^2,\cC_{0,X}) \,\, \mathrm{Fourier}\text{-}\mathrm{Mukai}\,\, \mathrm{equivalent} \,\, \mathrm{to}\,\,\perf(\bbP^2,\cC_{0,Y}) \Rightarrow X\simeq Y\,.
\end{equation*}
\item[(iii)] {\bf Quartic double solids:} let $X$ and $Y$ be quartic double solids. As proved by Kuznetsov \cite[\S 3]{kuznetfanothreefolds}, one has the following semi-orthogonal decompositions
\begin{eqnarray}\label{eq:semi-orthogonal}
\perf(X) =\langle \cT_X,\langle \cO_X,\cO_X(1)\rangle \rangle && \perf(Y)=\langle \cT_Y,\langle \cO_Y,\cO_Y(1)\rangle \rangle\,. 
\end{eqnarray}
Since $X$ and $Y$ are Fano (hence uniruled) threefolds, the pairings
\eqref{eq:pairings1} are non-degenerate. Moreover, $X$ and $Y$ are verepresentable (see Clemens or Tihomirov \cite{clemens4ic, tihoquarticsolid})
and the canonical bundle of the Fano variety $X$ is antiample. Now, suppose that there exists an equivalence $\phi: \cT_X \simeq \cT_Y$ of triangulated categories and that the composed functor $\Phi: \perf(X) \to \perf(Y)$ is of Fourier--Mukai type.
As a consequence, all the assumptions of Theorem~\ref{thm:main2} are satisfied.
Thanks to the work of Debarre \cite{debarre-quartic-double} and Voisin \cite{voisin-quartic-double},
the (generalized) Torelli theorem holds in this case. Hence, we obtain from Corollary~\ref{cor:Torelli} the following implication
\begin{equation}\label{eq:implication-iii}
\cT_X\simeq \cT_Y + \Phi \,\,\mathrm{Fourier}\text{-}\mathrm{Mukai} \,\,\Rightarrow X\simeq Y\,.
\end{equation}
\item[(iv)] {\bf Cubic threefolds:} let $X$ and $Y$ be cubic threefolds.
As proved by Kuznetsov \cite{kuznet-v14},
one has the following semi-orthogonal decompositions
\begin{eqnarray}\label{eq:semi-orthogonal2}
\perf(X) =\langle \cT_X,\langle \cO_X,\cO_X(1)\rangle \rangle && \perf(Y)=\langle \cT_Y,\langle \cO_Y,\cO_Y(1)\rangle \rangle\,. 
\end{eqnarray}
Similarly to item (iii), the pairings \eqref{eq:pairings1} are non-degenerate and the canonical bundle of $X$ is antiample. Now, suppose that there exists an equivalence $\phi: \cT_X \simeq \cT_Y$ of triangulated categories and that the composed functor $\Phi: \perf(X) \to \perf(Y)$ is of Fourier--Mukai type.
As a consequence, all the assumptions of Theorem~\ref{thm:main2} are satisfied. Thanks to the work of Clemens and Griffiths \cite{clemensgriffiths}, $X$ and $Y$ are verepresentable and the (generalized)
Torelli theorem holds. Hence, we obtain from Corollary~\ref{cor:Torelli} the implication \eqref{eq:implication-iii}.
\end{itemize}
Corollary~\ref{cor:Torelli} admits the following partial converse:
\begin{proposition}\label{prop:double-torelli}
Consider the following three cases:
\begin{itemize}
\item[(i)] $X$ and $Y$ quartic double solids (resp. cubic threefolds) and $\cT_X$ and $\cT_Y$ as in \eqref{eq:semi-orthogonal} (resp. as in \eqref{eq:semi-orthogonal2});
\item[(ii)] $X$ and $Y$ complete intersections of two even dimensional quadrics, $\cT_X:=\perf(\bbP^1,\cC_{0,X})\simeq \perf(C_X)$ and $\cT_Y:=\perf(\bbP^1,\cC_{0,Y})\simeq \perf(C_Y)$;
\item[(iii)] $X$ and $Y$ complete intersections of three odd dimensional quadrics,  $\cT_X:= \perf(\bbP^2,\cC_{0,X})$
and $\cT_Y:=\perf(\bbP^2,\cC_{0,Y})$.
\end{itemize}
In each one of the above three cases, the schemes $X$ and $Y$ are isomorphic if and only if the triangulated categories
$\cT_X$ and $\cT_Y$ are equivalent and the composed functor $\Phi: \perf(X) \to \perf(Y)$ is of Fourier--Mukai
type. Assuming Conjecture
\ref{conj:Kuznetsov1}, this happens if and only if $\cT_X$ is equivalent to $\cT_Y$.
\end{proposition}
\begin{remark}
Assuming Conjecture \ref{conj:Kuznetsov1}, Proposition \ref{prop:double-torelli} generalizes both implications of the main result
of \cite{berna_macri_mehro_stella}. Note that in the above item (ii) the composed functor $\Phi$ is always of Fourier--Mukai type.
\end{remark}
Let us consider now Example (E8). Motivated by Theorems \ref{thm:main1} and \ref{thm:main2}, we formulate the following conjecture:
\begin{conjecture}\label{conj:grass-pfaff}
Whenever $X$ and $Y$ are as in Example (E8), $Y$ is verepresentable and $J(Y) \simeq J(X)$ as
a principally polarized abelian variety.
\end{conjecture}
Assuming this conjecture, Proposition \ref{prop:double-torelli} gives rise to the following result:
\begin{corollary}\label{cor:tor-pf-gr}
Let $X$ and $Y$ be as in Example (E8). Assuming Conjecture \ref{conj:grass-pfaff}, the generalized and the categorical
Torelli theorem are equivalent.
\end{corollary}
\section{Application II: Kuznetsov's conjecture on Fano threefolds}
Let $X_{d'}$ and $Y_d$ be two Fano threefolds of Picard number one, indexes $1$ and $2$, and degrees $d'$ and $d$, respectively. As proved by Kuznetsov \cite[Corollary 3.5 and Lemma~3.6]{kuznetfanothreefolds}, whenever $d'\equiv 2$ modulo $4$, one has the following semi-orthogonal decompositions
\begin{eqnarray*}
\perf(X_{d'})=\langle \cT_{X_{d'}}, \langle \cE_{X_{d'}}, \cO_{X_{d'}} \rangle \rangle && \perf(Y_d)=\langle \cT_{Y_d} ,\langle \cO_{Y_d},\cO_{Y_d}(1) \rangle \rangle\,,
\end{eqnarray*}
where $\cE_{X_{d'}}$ is an exceptional vector bundle of rank $2$.
\begin{conjecture}{(Kuznetsov \cite[Conjecture~3.7]{kuznetfanothreefolds})}\label{conj:Kuznetsov2}
Let $\cM\cF^i_d$ be the moduli space of Fano threefolds with Picard number one, index $i$ and degree $d$. Under these notations, there exists a correspondence $Z_d \subset \cM\cF^2_d \times \cM\cF^1_{4d+2}$
which is dominant over each factor. Moreover, at each point $(Y_d , X_{4d+2} )$ of this correspondence $Z_d$ one has an equivalence $\cT_{X_{4d+2}} \simeq \cT_{Y_d}$ of triangulated categories.
\end{conjecture}
Kuznetsov proved this conjecture when $d=3,4,5$. Making use of Theorem~\ref{thm:main2}, we show that the case $d=2$ is {\em not} compatible with Conjecture~\ref{conj:Kuznetsov1}.
\begin{theorem}\label{thm:conjectures}
Whenever Conjecture \ref{conj:Kuznetsov1} holds, the case $d=2$ of Conjecture \ref{conj:Kuznetsov2} is false.
\end{theorem}
\begin{remark}
As the proof of Theorem~\ref{thm:conjectures} suggests (see \S\ref{sec:proofs}), the assumption in Conjecture \ref{conj:Kuznetsov2} that the correspondence $Z_d$ is dominant over each factor should be removed.
\end{remark}
\section{Application III: Blow-ups}
Making use of Theorem~\ref{thm:main2}, one obtains a categorical proof of the following result concerning blow-ups.
The original proof, due to Clemens-Griffiths \cite[\S 3]{clemensgriffiths}, is rather different since it is based on (co)homological methods.

\begin{theorem}\label{thm:blow-ups}
Let $X$ be a verepresentable
$\bbC$-scheme of dimension $3$ with ample, antiample, or trivial canonical bundle,
$\widetilde{X_{\mathrm{pt}}} \to X$ the blow-up of $X$ at a point, and $\widetilde{X_C} \to X$ the blow-up of $X$
along a smooth curve $C$. Under these assumptions, $J(\widetilde{X_{\mathrm{pt}}})\simeq J(X)$ and $J(\widetilde{X_C}) \simeq J(X) \times J(C)$ as principally polarized abelian varieties.
\end{theorem}
\begin{remark}
The isomorphism $J(\widetilde{X_{\mathrm{pt}}})\simeq J(X)$ holds more generally in every dimension as long as one assumes
that $\widetilde{X_{\mathrm{pt}}}$ is also verepresentable and that the pairings \eqref{eq:pairings1} are non-degenerate for $X$.
\end{remark}

\section{Application IV: Quadric fibrations and intersection of quadrics}\label{sec:app-4}

Let $S$ be a $\bbC$-scheme and $Q \to S$ a flat quadric fibration of relative
dimension $n$. Out of this data,
we can construct the sheaf $\cC_0$ on $S$ of the even parts of Clifford algebras and the derived category $\perf(S,\cC_0)$
of perfect complexes of $\cC_0$-algebras; consult Kuznetsov \cite[\S 3]{kuznetquadrics} for details. As proved by Kuznetsov \cite[Thm.~4.2]{kuznetquadrics},
we have the following semi-orthogonal decomposition
\begin{equation}\label{eq:mot-decomp-1}
\perf(Q) = \langle \perf(S,\cC_0),  \langle \underbrace{\perf(S),\ldots, \perf(S)}_{n\text{-}\mathrm{factors}}\rangle\rangle\,.
\end{equation}
Now, let $T$ be a $\bbC$-scheme, $X\to T$ a generic relative complete intersection of $r+1$ quadric hypersurfaces of dimension $n$,
and $Q \to S$ the associated linear span; consult \cite{auel-bernar-bologn}. As proved in {\em loc. cit.}, the following holds:
\begin{itemize}
\item[(i)] When $2r<n$, the fibers of $X$ are of Fano type and, thanks to homological projective duality, we have the following semi-orthogonal decomposition
\begin{equation}\label{hpd_Fano} 
\perf(X) = \langle \perf(S,\cC_0),  \langle \underbrace{\perf(T),\ldots, \perf(T)}_{(n-2r)\text{-}\mathrm{factors}}\rangle\rangle\,.
\end{equation}
\item[(ii)] When $2r=n$, the fibers of $X$ are generically of Calabi-Yau type and, thanks once again to homological projective duality, we have $\perf(X)\simeq \perf(S,\cC_0)$.
\item[(iii)] When $2r>n$, the fibers of $X$ are generically of general type and there exists a fully faithful functor
$\perf(X) \hookrightarrow \perf(S,\cC_0)$.
\end{itemize}
In what follows, we will treat separately the cases (i)-(ii) from (iii). In the cases (i)-(ii), the orthogonal complement of $\perf(S,\cC_0)$ is well understood.
Moreover, all the known examples where $X$ and $Q$ are verepresentable are of this type; see \S\ref{sect-verepre-quadric}.
In the case (iii), the orthogonal complement of $\perf(S,\cC_0)$ in $\perf(X)$ is less well-understood, and we will only briefly treat in \S\ref{sect-quadric-gentype}.
\subsection{Reduction by hyperbolic splitting}
Let $S$ be a $\bbC$-scheme and $Q\to S$ a flat quadric fibration
of relative dimension $n \geq 2$.
Whenever such fibration admits a {\em regular section}, \ie a section cutting
a regular point at each fiber, one can perform its {\em reduction by hyperbolic splitting} (along this regular section); consult
\cite[\S1]{auel-bernar-bologn} for details. Roughly speaking, one takes the base of the cone given (fiberwise) by the intersection of the quadric with the tangent
space at the section. We obtain in this way another quadric fibration $Q'\to S$
of relative dimension $n-2$, and another sheaf $\cC_0'$ on $S$ of the even parts
of Clifford algebras. As proved in \cite[Thm.~1.26]{auel-bernar-bologn}, one has an equivalence of categories $\perf(S,\cC_0)
\simeq \perf(S,\cC_0')$ (given by a Fourier-Mukai functor) whenever $Q\to S$ is a generic quadric fibration and $Q$ is smooth.
By combining this result with Theorem~\ref{thm:main1} one then obtains:
\begin{theorem}\label{thm:morita-hyperbolic}
Let $Q\to S$ and $Q'\to S$ be as above. Assume that the bilinear pairings \eqref{eq:pairings1} (associated to $Q$ and $Q'$)
are non-degenerate and that ${\bf J}(\perf^\dg(S))=0$. Under these assumptions, there exist inverse isogenies $\tau: \prod_{i=0}^{d_Q-1}J^i_a(Q) \to \prod_{i=0}^{d_{Q'}-1}J^i_a(Q')$ and $\sigma:\prod_{i=0}^{d_{Q'}-1} J^i_a(Q') \to \prod_{i=0}^{d_Q-1}J^i_a(Q)$.
\end{theorem}
\begin{remark}\label{rmk-non-deg-pair1} The bilinear pairings \eqref{eq:pairings1} 
are known to be non-degenerate whenever ${\mathrm{dim}}(S) \leq 2$; see Vial \cite[Thm.~7.4]{vial-fibrations}.
When $\mathrm{dim}(S)\leq 2$ and $S$ is a projective space of a rational surface, the triangulated category $\perf(S)$
is known to have a full exceptional collection; see \cite[\S 4, 4.1]{bolognesi_bernardara:representability}. Consequently, we have ${\bf J}(\perf^\dg(S))=0$.
\end{remark}

Note that whenever $Q \to S$ is a quadric fibration, the relative canonical
bundle $\omega_{Q/S}$ is antiample. Hence, the canonical bundle $\omega_Q$ can never be trivial or ample.
By combining Theorems \ref{thm:main2} and \ref{thm:morita-hyperbolic}, one then obtains the following result (whose
applications will be discussed in \S\ref{sect-verepre-quadric}):
\begin{corollary}\label{cor-polar-hyper}
Let $Q \to S$ and $Q' \to S$ be as in Theorem \ref{thm:morita-hyperbolic} with ${\bf J}(\perf^\dg(S))=0$.
Suppose moreover that $Q$ and $Q'$ are
verepresentable and that the canonical bundle of $Q$ is antiample.
Under these assumptions, $J(Q) \simeq J(Q')$ as principally polarized abelian varieties.
\end{corollary}

\subsection{Complete intersections of Fano and Calabi--Yau type}
Let $T$ be a $\bbC$-scheme, $X \to T$ a
generic relative complete intersection of $r+1$ quadric hypersurfaces of dimension $n$,
and $Q \to S$ the associated linear span; see \cite{auel-bernar-bologn}. Suppose that $2r \leq n$. This implies that the fibers of $X$ are either of Fano ($2r < n$) or Calabi--Yau ($2r=n$) type.
By combining (relative) homological projective duality results with Theorem~\ref{thm:main1}, one then obtains:
\begin{theorem}\label{thm:hpd-quadr}
Let $X \to T$ be as above (with $2r\leq n$) and $Q\to S$ the associated linear span. Assume that the bilinear pairings
\eqref{eq:pairings1} (associated to $X$ and $Q$) are non-degenerate, and that ${\bf J}(\perf^\dg(T))=0$. 
Under these assumptions, there exist inverse isogenies $\tau:\prod_{i=0}^{d_X-1}J^i_a(X) \to \prod_{i=0}^{d_Q-1} J^i_a(Q)$ and $\sigma: \prod_{i=0}^{d_Q-1}J^i_a(Q) \to \prod_{i=0}^{d_X-1}J^i_a(X)$. 
\end{theorem}

\begin{remark}\label{rmk-non-deg-pair2} The bilinear pairings \eqref{eq:pairings1} are known to be non-degenerate
when $r=0$ and ${\mathrm{dim}}(S) \leq 2$ (see Vial \cite[Thm.~7.4]{vial-fibrations}),
when $r=1$, ${\mathrm{dim}}(S) \leq 2$ and $d_X \leq 6$ (see Vial \cite[Thm.~7.6]{vial-fibrations}), 
and when $r$ is arbitrary and $S$ is a point. Note that since $S\to T$ is a projective bundle,
${\bf J}(\perf^\dg(T))=0 \Rightarrow {\bf J}(\perf^\dg(S))=0$.
\end{remark}

\subsection{Verepresentability of quadrics, intersections of quadrics and their fibrations}\label{sect-verepre-quadric}
Let $X \to T$ be a fibration in complete intersection of $r+1$ quadrics of relative dimension $n$.
In the following three cases $X$ is verepresentable and its canonical bundle is either ample, antiample, or trivial; see Examples \ref{ex:very2}-\ref{ex:very3} also.
\begin{itemize}
\item[(i)] ($T=\mathrm{point}$) In this case $X$ is the intersection of $r+1$ quadric hypersurfaces in $\bbP^{n+1}$.
When $r=0$, $X$ is just a quadric. Hence, it is known to be
verepresentable for all $n$. When $r=1$, $X$ is verepresentable for
$n$ even (see see \cite[Thm. 1.5]{marcello-goncalo-chowgroups}, and Reid \cite{reid:thesis}) and its canonical bundle is either trivial ($n=2$) or
antiample ($n>2$). When $r=2$, $X$ is verepresentable for $n>3$ odd 
(see \cite[Thm. 1.5]{marcello-goncalo-chowgroups}, and Beauville \cite[\S 6]{beauvilleprym}) and its canonical bundle is antiample. When $r=n \geq 3$, $X$ is a curve of genus $g>1$ in $\bbP^{n+1}$. Hence, it is verepresentable and its canonical bundle is ample.
\item[(ii)]($T=\bbP^1$)
In all these cases, the canonical bundle of $X$ is antiample.
When $r=0$, $X \to \bbP^1$ is a quadric fibration. If $n$ is odd, then $X$ is verepresentable; see Vial \cite[\S 4]{vial-fibrations}. Thanks to the work of Kuznetsov \cite[Corollary 3.16]{kuznetquadrics} and Polishchuk \cite{polishchuk}, $\perf(\bbP^1,\cC_0)$ is generated by exceptional objects.
If $n$ is even, then $X$ is also verepresentable. This follows from the combination of Vial's motivic description \cite[\S 4]{vial-fibrations} with Reid's work \cite{reid:thesis} on the intermediate Jacobian.
When $r=1$ and $n=3$, $X \to \bbP^1$ is a fibration in del Pezzo surfaces of degree 4. Thanks to the work of Gorchinskiy-Guletskii~\cite{gorch-gul-motives-and-repr} and Kanev~\cite{kanevdp1}, $X$ is also verepresentable.
\item[(iii)]($T=\mathrm{rational\,\,surface}$)
In all these cases, the canonical bundle of $X$ is antiample.
When $r=0$ and $n=1$, $X$ is a conic bundle over $T$. Hence, it is verepresentable; see Beauville and others~\cite{beauvilleprym,beltrachow,bolognesi_bernardara:conic_bundles}. When $r=0$, $n$ is odd and $T=\bbP^2$, $X$
is also verepresentable; see Beauville \cite{beauvilleprym}. The same proof should also work for every rational surface.
\end{itemize}
Now, let $Q\to S$ be the linear span associated to the above fibration $X \to T$. Assume that $r\neq 0$ (otherwise $X\simeq Q$). In the following cases $Q$ is verepresentable:
\begin{itemize}
\item[(i')] ($T=\mathrm{point}$, $r=1$, $n$ even). In this case $X$ is the intersection of two even dimensional
quadrics and $Q \to \bbP^1$ is an even dimensional fibration. Hence, from the above item (ii) one concludes that $Q$ is verepresentable.

\item[(ii')] ($T=\mathrm{point}$, $r=2$, $n$ odd). In this case $X$ is the intersection of three odd dimensional
quadrics and $Q \to \bbP^2$ is an odd dimensional fibration. Hence, from the above item (iii) one concludes that $Q$ is verepresentable.

\item[(iii')] ($T = \bbP^1$, $r=1$, $n=3$). In this case $X \to \bbP^1$ is a del Pezzo fibration of degree 4 and $Q \to S$ is a quadric fibration over an Hirzebruch surface. Hence, from the above (conditional) item (iii) one concludes that $Q$ is verepresentable.
\end{itemize}
By combining Theorems \ref{thm:main2} and \ref{thm:hpd-quadr}, one then obtains the following result:
\begin{corollary}\label{cor-polar-inter}
Let $X \to T$ and $Q \to S$ be as in Theorem \ref{thm:hpd-quadr} with ${\bf J}(\perf^{\dg}(T))=0$. Assume moreover that $X$ and $Q$ are verepresentable
and that the canonical bundle of $X$ is antiample or trivial. This holds for instance in the above cases (i')-(iii').
Under these assumptions, $J(X) \simeq J(Q)$ as principally polarized abelian varieties.
\end{corollary}
Suppose now that $Q \to S$ admits a regular section. Let us write $Q' \to S$ for its reduction by hyperbolic splitting.
If $Q$ belongs to one of the above cases (i')-(iii'), then so it does $Q'$. Moreover, when $S=\bbP^1$, $Q \to S$ admits a section
thanks to Tsen's Theorem; see Koll{\`a}r \cite[Thm.~IV.6.5]{kollar-book}. Furthermore, when $S$ is a rational surface,
$Q \to S$ also admits a section in the particular case of relative dimension $n>2$; see Beauville \cite[\S IV]{beauvilleprym}.
\begin{corollary}\label{cor-alltogether}
Let $X \to T$ and $Q \to S$ be as in Theorem \ref{thm:hpd-quadr} with ${\bf J}(\perf^{\dg}(T))=0$. 
Assume that $Q\to S$ admits a regular section, and let $Q' \to S$
be its reduction by hyperbolic splitting. Assume moreover that $X$ and $Q$ are verepresentable
and that the canonical bundle of $X$ is antiample or trivial. This holds for instance in the above cases (i')-(iii').
Under these assumptions, $J(X) \simeq J(Q')$ as principally polarized abelian varieties.
\end{corollary}
\begin{example}
In the case $T=\bbP^1$, consider $X \to \bbP^1$ as a del Pezzo fibration of degree 4, \ie  $r=1$ and $n=3$. Let $Y \to S$ be
the conic bundle (birational to $X$) constructed by Alexeev~\cite{alekseev:dP4}. From Corollary~\ref{cor-alltogether}
we obtain then the isomorphism $J(X) \simeq J(Q')$ of principally polarized abelian varieties described by Alexeev \cite{alekseev:dP4}.
\end{example}
\subsection{Complete intersections of general type}\label{sect-quadric-gentype}
Let $T$ be a $\bbC$-scheme, $X \to T$ a
generic relative complete intersection of $r+1$ quadric hypersurfaces of dimension $n$,
and $Q \to S$ the associated linear span; consult \cite{auel-bernar-bologn} for details. Assume that $2r>n$.
In this case the fibers of $X$ are of general type. Recall from above that we have a fully faithful functor $\perf(X) \to \perf(S,\cC_{0,X})$ of Fourier-Mukai type. Thanks to Kuznetsov and others (see \cite{auel-bernar-bologn,kuznetquadrics}) the orthogonal complement can be described in terms of higher Clifford modules. By combining (relative) homological projective duality results with Theorem~\ref{thm:main1}, one then obtains:
\begin{theorem}\label{thm:hpd-gentype}
Let $X \to T$ be as above (with $2r > n$) and $Q\to S$ the associated linear span. Assume that the bilinear pairings
\eqref{eq:pairings1} (associated to $X$ and $Q$) are non-degenerate. 
Under these assumptions, there exist well-defined isogenies $\tau:\prod_{i=0}^{d_X-1}J^i_a(X) \to \prod_{i=0}^{d_Q-1} J^i_a(Q)$ and $\sigma: \prod_{i=0}^{d_Q-1}J^i_a(Q) \to \prod_{i=0}^{d_X-1}J^i_a(X)$, where $\tau$ is injective
as a morphism in $\Ab(\bbC)_\bbQ$ and $\sigma$ is its retraction.
\end{theorem}
To the best of the authors knowledge, $X$ is known to be verepresentable only when $T$ is a point and $r=n\geq 3$. In this particular case, $X$ is curve, complete intersection of $r+1$ quadrics in $\bbP^{r+1}$, and $Q \to \bbP^{r}$ is a quadric
fibration of relative dimension $r$. On the other hand, very little is known about $Q$. For instance, it is not know if $Q$ is verepresentable neither if the bilinear pairings \eqref{eq:pairings1} are non-degenerate. Nevertheless, by
combining Kuznetsov's results \cite[Thms.~4.2 and 5.5]{kuznetquadrics} with the fact that $\perf(\bbP^r)$ is generated by
exceptional objects, we obtain the following semiorthogonal decomposition 
$$\perf(Q) = \langle \cC_{-r}, \ldots \cC_{-1}, \perf(X), E_1, \ldots, E_{r^2+r} \rangle\,,$$
where the $E_{-i}$'s are exceptional objects and the $\cC_i$'s are the Clifford modules
introduced by Kuznetsov in \cite[\S 3]{kuznetquadrics}. Using Theorem \ref{thm:hpd-gentype}, this gives rise
to the following (conditional) corollary:
\begin{corollary}
Let $X$ be a curve, complete intersection of $r+1$ quadrics in $\bbP^{r+1}$, and $Q \to \bbP^r$
the associated linear span. Assume that the pairings \eqref{eq:pairings1} (associated to $Q$) are non-degenerate and that the
modules $\cC_{-i}$ are exceptional objects in $\perf(\bbP^r,\cC_0)$. Under these assumptions, there is a well-defined
isogeny $\tau: J(X) \to \prod_{i=0}^{d_Q-1}J^i_a(Q)$. Assume moreover that $Q$ is verepresentable. Under this extra assumption, the map $\tau: J(X) \to J(Q)$ preserves the principal polarization.
Furthermore, if the modules $\cC_{-i}$ are exceptional objects in $\perf(\bbP^r,\cC_0)$, then $\tau$ becomes an isomorphism $J(X) \simeq J(Q)$
of principally polarized abelian varieties. 
\end{corollary}
\section{Background on dg categories}\label{sec:dg}
Let $\cC(\bbC)$ be the category of cochain complexes of $\bbC$-vector spaces; we use cohomological notation. A {\em differential graded
(=dg) category $\cA$} is a category enriched over $\cC(\bbC)$ (morphisms sets $\cA(x,y)$ are complexes) in such a way that the composition law fulfills the Leibniz rule $d(f \circ g) =d(f) \circ g +(-1)^{\mathrm{deg}(f)}f \circ d(g)$. A {\em dg functor} $F:\cA\to \cB$ is a
functor enriched over $\cC(\bbC)$; consult Keller's ICM address \cite{ICM-Keller} for further details. In what follows we will write
$\dgcat(\bbC)$ for the category of (small) dg categories and dg functors.  The {\em tensor product $\cA\otimes\cB$} of two dg categories $\cA$ and $\cB$ is defined
as follows: the set of objects is the cartesian product of the sets of objects of $\cA$ and $\cB$ and the complexes of morphisms are given by $(\cA\otimes\cB)((x,w),(y,z)):= \cA(x,y) \otimes \cB(w,z)$.
\subsection*{(Bi)modules}
Let $\cA$ be a dg category. Its {\em opposite} dg category $\cA^\op$ has the same objects and complexes of morphisms given by
$\cA^\op(x,y):=\cA(y,x)$. A {\em right $\cA$-module} $M$ is a dg functor $M:\cA^\op \to \cC^\dg(\bbC)$ with values in the dg category $\cC^\dg(\bbC)$ of cochain complexes of $\bbC$-vector spaces; see Keller \cite[\S2.3]{ICM-Keller}. We will write $\cC(\cA)$ for the category of right $\cA$-modules. Recall again from Keller \cite[\S3.2]{ICM-Keller} that the {\em derived
category $\cD(\cA)$ of $\cA$} is the localization of $\cC(\cA)$ with respect to the class of objectwise quasi-isomorphisms. Its full subcategory of compact objects will be denoted by $\cD_c(\cA)$.

Now, let $\cA$ and $\cB$ be two dg categories. A {\em $\cA\text{-}\cB$-bimodule $B$} is a dg functor $B:\cA \otimes \cB^\op\to \cC^\dg(\bbC)$, \ie a right $(\cA^\op \otimes \cB)$-module. A standard example is given by the $\cA\text{-}\cA$-bimodule
\begin{eqnarray}\label{eq:bimodule1}
\cA\otimes\cA^\op \too \cC^\dg(\bbC) && (x,y) \mapsto \cA(y,x)\,.
\end{eqnarray} 
Let us denote by $\rep(\cA,\cB)$ the full triangulated subcategory of $\cD(\cA^\op \otimes \cB)$ consisting of those $\cA\text{-}\cB$-bimodules
$B$ such that for every object $x \in \cA$ the right $\cB$-module $B(x,-)$ belongs to $\cD_c(\cB)$. Note that every dg functor
$F:\cA \to \cB$ gives rise to a $\cA\text{-}\cB$-bimodule
\begin{eqnarray*}
{}_FB:\cA\otimes \cB^\op \too \cC^\dg(\bbC) && (x,w) \mapsto \cB(w,F(x))
\end{eqnarray*}
which belongs to $\rep(\cA,\cB)$.
\subsection*{Morita equivalences}
A dg functor $F:\cA\to \cB$ is called a {\em Morita equivalence} if the restriction of scalars functor $\cD(\cB) \stackrel{\simeq}{\to} \cD(\cA)$
is an equivalence of (triangulated) categories; see Keller \cite[\S4.6]{ICM-Keller}. As proved in \cite[Thm.~5.3]{IMRN}, the category $\dgcat(\bbC)$
carries a Quillen model category whose weak equivalences are the Morita equivalences. Let us write $\Hmo(\bbC)$ for the homotopy category
obtained. As proved in {\em loc. cit.}, the assignment $F \mapsto {}_F B$ gives rise to a bijection 
\begin{equation}\label{eq:bij}
\Hom_{\Hmo(\bbC)}(\cA,\cB)\simeq \mathrm{Iso}\,\rep(\cA,\cB)\,,
\end{equation}
where $\mathrm{Iso}$ stands for the set of isomorphism classes. Moreover, under \eqref{eq:bij}, the composition law in $\Hmo(\bbC)$ corresponds
to the derived tensor product of bimodules. 
\subsection*{Pretriangulated dg categories}
Let $\cA$ be a dg category. The $\bbC$-linear category $\dgHo(\cA)$ has the same objects as $\cA$ and morphisms given by $\dgHo(\cA)(x,y):=H^0(\cA(x,y))$, where $H^0(-)$ is the $0^{\mathrm{th}}$ cohomology group functor. The dg category $\cA$ is called {\em pretriangulated} if $\dgHo(\cA)$ is a triangulated category; see Keller \cite[\S4.5]{ICM-Keller}.
\section{Background on noncommutative Chow motives}\label{sec:NCChow}
In this section we recall the construction of the category of noncommutative Chow motives. For further details we invite the reader to consult the survey article \cite{survey}. Recall from \S\ref{sec:dg} that one has a well-defined functor
\begin{eqnarray}\label{eq:functor11}
\dgcat(\bbC) \too \Hmo(\bbC) && F \mapsto {}_FB\,.
\end{eqnarray}
The {\em additivization of $\Hmo(\bbC)$} is the additive category $\Hmo_0(\bbC)$ which has the same objects as $\Hmo(\bbC)$ and abelian groups of morphisms given by $\Hom_{\Hmo_0(\bbC)}(\cA,\cB):=K_0\,\rep(\cA,\cB)$, where $K_0$ stands for the Grothendieck group of the triangulated category $\rep(\cA,\cB)$. The composition law is induced by the tensor product of bimodules; consult \cite[\S6]{IMRN} for further details. Note that we have a canonical functor
\begin{eqnarray}\label{eq:functor2}
\Hmo(\bbC) \too \Hmo_0(\bbC) && B \mapsto [B]\,.
\end{eqnarray}
The {\em $\bbQ$-linearization} of $\Hmo_0(\bbC)$ is the $\bbQ$-linear additive category $\Hmo_0(\bbC)_\bbQ$ obtained by tensoring each abelian group of morphisms of $\Hmo_0(\bbC)$ with $\bbQ$. Note that one has also a canonical functor
\begin{eqnarray}\label{eq:functor3}
\Hmo_0(\bbC) \too \Hmo_0(\bbC)_\bbQ && [B] \mapsto [B]\otimes_\bbZ \bbQ\,.
\end{eqnarray}
Since the three functors \eqref{eq:functor11}-\eqref{eq:functor3} are the identity on objects we will make no notational distinction between a dg category and its image in $\Hmo_0(\bbC)_\bbQ$.

Now, recall from Kontsevich \cite{IAS,Miami,finMot} that a dg category $\cA$ is called {\em smooth} if the above $\cA\text{-}\cA$-bimodule
\eqref{eq:bimodule1} belongs to $\cD_c(\cA^\op\otimes \cA)$ and {\em proper} if for each ordered pair of objects $(x,y)$ the inequality
$\sum_i \mathrm{dim}_\bbC H^i\cA(x,y)< \infty$ holds. Examples include finite dimensional $\bbC$-algebras of finite global dimension (\eg\ path algebras of finite quivers without oriented cycles) as well as the unique dg enhancements $\perf^\dg(X)$ of the derived categories $\perf(X)$ of perfect complexes; see Lunts-Orlov \cite{LO}. The category $\NChow(\bbC)_\bbQ$ of {\em noncommutative Chow motives} (with rational coefficients) is then by definition the pseudo-abelian envelope of the full subcategory of $\Hmo_0(\bbC)_\bbQ$ consisting of the smooth and proper dg categories. For the relationship between $\NChow(\bbC)_\bbQ$ and the classical category $\Chow(\bbC)_\bbQ$ of Chow motives we invite the reader to consult \cite{CvsNC}.
\section{Proof of Theorem~\ref{thm:main1}}
By assumption there exist semi-orthogonal decompositions $\perf(X)=\langle \cT_X,\cT_X^\perp\rangle$ and $ \perf(Y)=\langle \cT_Y,\cT_Y^\perp\rangle$ and an equivalence $\phi: \cT_X \stackrel{\simeq}{\to} \cT_Y$ of triangulated categories. Out of this data one constructs the composed functor
\begin{equation}\label{eq:composed-functor}
\Phi: \perf(X) \stackrel{\pi_X}{\too}\cT_X \stackrel{\phi}{\too} \cT_Y \stackrel{i_Y}{\too} \perf(Y)\,,
\end{equation}
where $\pi_X$ stands for the projection and $i_Y$ for the inclusion. Once again by assumption, $\Phi$ is of Fourier-Mukai type,
\ie there exists a perfect complex $\cE \in \perf(X\times Y)$ such that $\Phi$ is isomorphic to the Fourier-Mukai functor
$\Phi_\cE:= Rq_* (p^* (-) \otimes^{L} \cE)$; the complex $\cE$ is usually called the {\em kernel} of $\Phi_\cE$. 
Recall from Huybrechts \cite[Proposition~5.9]{huybrechts} that the right adjoint of $\Phi_\cE$ is also of Fourier-Mukai type and that its 
kernel is given by $\cE_R:=\cE^\vee \otimes^L p^\ast \omega_X[d_X]$, where $\cE^\vee$ is the dual of $\cE$ and $\omega_X[d_X]$ is the canonical line bundle of $X$ shifted by the dimension of $X$. Moreover, the unit $\gamma: \Id \Rightarrow \Phi_{\cE_R} \circ \Phi_\cE$ 
of this adjunction is an isomorphism when evaluated at any object of $\cT_X \subset \perf(X)$;
see Kuznetsov \cite[Thm.~3.3]{kuznetsov:hpd}.

Now, recall that the triangulated categories $\perf(X)$ and $\perf(Y)$ admit unique dg enhancements $\perf^\dg(X)$ and $\perf^\dg(Y)$. As proved by To{\"e}n \cite[\S8.3]{Toen}, every perfect complex $\cE \in \perf(X\times Y)$ gives rise to a dg functor
\begin{eqnarray}\label{eq:dg-enhancement}
\Phi^\dg_\cE: \perf^\dg(X) \too \perf^\dg(Y) && \cF \mapsto Rq_\ast(p^\ast(\cF) \otimes^L \cE)
\end{eqnarray}
such that $\dgHo(\Phi_\cE^\dg)\simeq \Phi_\cE$. Moreover, one has the following bijection
\begin{eqnarray}\label{eq:bijection-kernels1}
& \mathrm{Iso}\,\perf(X\times Y) \stackrel{\simeq}{\too} \Hom_{\Hmo(\bbC)}(\perf^\dg(X),\perf^\dg(Y)) & \cE \mapsto \Phi^\dg_\cE\,.
\end{eqnarray}
The following result is known to the experts, but to the best of the authors' knowledge, no proof can be found in the literature.
\begin{proposition}\label{prop:FMequalsdg}
The following conditions are equivalent:
\begin{itemize}
\item[(i)] (Kuznetsov's conjecture) The functor \eqref{eq:composed-functor} is of Fourier-Mukai type;
\item[(ii)] The functor \eqref{eq:composed-functor} admits a {\em dg enhancement}, \ie there exists a dg functor $\Phi^\dg:\perf^\dg(X) \to \perf^\dg(Y)$ such that $\dgHo(\Phi^\dg)\simeq \Phi$.
\end{itemize}
\end{proposition}
\begin{proof}
Assume that $\Phi$ is of Fourier-Mukai type with kernel $\cE$. Then, take for $\Phi^\dg$ the dg functor \eqref{eq:dg-enhancement}. This shows the implication (i)$\Rightarrow$(ii). Assume now that $\Phi$ admits a dg enhancement $\Phi^\dg$. Thanks to bijection \eqref{eq:bijection-kernels1}, $\Phi^\dg$ is of the form $\Phi_\cE^\dg$ for a certain perfect complex $\cE \in \perf(X\times Y)$. Using the equivalence $\dgHo(\Phi^\dg)\simeq \Phi$, one then concludes that $\Phi$ is of the form $\Phi_\cE$ (and hence of Fourier-Mukai type). This shows the implication (ii)$\Rightarrow$(i).
\end{proof}
Let us denote by $i_X^\dg$ the inclusion of dg categories $\cT_X^\dg \to \perf^\dg(X)$. The above projection functor $\pi_X: \perf(X) \to \cT_X$ also admits a dg enhancement:
\begin{lemma}\label{lem:projection}
There exists a well-defined morphism $\pi_X^\dg:\perf^\dg(X) \to \cT_X^\dg$ in $\Hmo(\bbC)$ such that $\pi_X^\dg\circ i_X^\dg=\Id$.
\end{lemma}
\begin{proof}
Recall first that we have the following inclusions of dg categories
\begin{eqnarray}
i_X^\dg: \cT_X^\dg \to \perf^\dg(X) && i_X^{\perp,\dg}: \cT_X^{\perp,\dg} \to \perf^\dg(X)\,.
\end{eqnarray}
Their intersection in $\perf^\dg(X)$ is the zero object. Hence, let us denote by $T$ the full dg subcategory of $\perf^\dg(X)$ consisting of those objects that belong to $\cT_X^\dg$ or to $\cT_X^{\perp,\dg}$. Note that the dg functor $i_X^\dg$ factors through the inclusion $T\subset \perf^\dg(X)$. Since by hypothesis one has a semi-orthogonal decomposition $\perf(X)=\langle \cT_X,\cT_X^\perp\rangle$, the objects of $\dgHo(T)$ form a set of generators of $\perf(X)$. Consequently, the inclusion of dg categories $T \subset \perf^\dg(X)$ is a Morita equivalence (see Keller \cite[Lemma~3.10]{ICM-Keller}) and hence an isomorphism in $\Hmo(\bbC)$. Now, clearly one has a well-defined dg functor $\pi_X^\dg:T \to \cT_X^\dg$ that is the identity of $\cT_X^\dg$ and which sends all the remaining objects to zero. This achieves the proof. 
\end{proof}
\begin{notation}
Recall from Keller \cite[\S2.3]{ICM-Keller} that given two dg functors $F,G: \cA \to \cB$, the {\em complex of morphisms} $\underline{\mathrm{Hom}}(F,G)$ has as its $n^{\mathrm{th}}$ component the $\bbC$-vector space formed by the families of morphisms $\phi_x \in \cB(F(x),G(x))^n$ of degree $n$ such that $G(f) \circ \phi_x = \phi_x \circ F(f)$ for all $f \in \cA(x,y)$ and $x,y \in \cA$. The differential is induced by that of $\cB(F(x),G(x))$. The set of {\em morphisms} $\nu$ from $F$ to $G$ is by definition in bijection with $Z^0 \underline{\mathrm{Hom}}(F,G)$, where $Z^0(-)$ denotes the degree zero cycles functor.
\end{notation}
\begin{lemma}\label{lem:quasi-bimodules}
Let $F,G: \cA \to \cB$ be two dg functors between pretriangulated dg categories and $\nu: F \Rightarrow G$ a morphism from $F$ to $G$. This data gives naturally rise to a morphism between $\cA\text{-}\cB$-bimodules ${}_\nu B:{}_F B \Rightarrow {}_G B$ and to a natural transformation $\dgHo(\nu):\dgHo(F) \Rightarrow \dgHo(G)$ between triangulated functors. Under these notations, whenever $\dgHo(\nu)$ is an isomorphism, ${}_\nu B$ is a quasi-isomorphism.
\end{lemma}
\begin{proof}
Let $x \in \cA$ and $w \in \cB$. One needs to prove that the induced homomorphisms
\begin{eqnarray}\label{eq:induced1}
(\nu_x)_\ast: H^i\cB(w,F(x)) \to H^i\cB(w,G(x)) && i \in \bbZ
\end{eqnarray}
are isomorphisms. Recall that $\cA$ and $\cB$ are pretriangulated. Hence, \eqref{eq:induced1} identifies with the induced homomorphisms
\begin{eqnarray}\label{eq:induced2}
&(\dgHo(\eta)_x)_\ast: \Hom_{\dgHo(\cB)}(w[i],F(x)) \to \Hom_{\dgHo(\cB)}(w[i],G(x)) & i \in \bbZ\,,
\end{eqnarray}
where $w[i]$ is the $i^{\mathrm{th}}$ suspension of $z$ in the triangulated category $\dgHo(\cB)$. Since by hypothesis $\dgHo(\eta)$ is an isomorphism, one then concludes that \eqref{eq:induced2} (and hence that \eqref{eq:induced1}) is an isomorphism. This achieves the proof.
\end{proof}
\begin{proposition}
One has the following commutative diagram
\begin{equation}\label{eq:com2}
\xymatrix{
\perf^\dg(X) \ar[r]^-{\Phi_\cE^\dg} & \perf^\dg(Y) \ar[r]^-{\Phi_{\cE_R}^\dg} & \perf^\dg(X) \\
\cT_X^\dg \ar[u]^-{i_X^\dg} \ar@{=}[rr] && \cT_X^\dg \ar[u]_-{i_X^\dg}
}
\end{equation}
in the homotopy category $\Hmo(\bbC)$.
\end{proposition}
\begin{proof}
Recall from Caldararu-Willerton~\cite{Caldararu} the construction of the $2$-category $\Var(\bbC)$ of integral kernels. The objects are the smooth projective $\bbC$-schemes, the categories of morphisms are given by $\Hom_{\Var(\bbC)}(X,Y):=\perf(X\times Y)$, the composition law is induced by the convolution of kernels, and the identity of every object $X$ is the structure sheaf $R\Delta_\ast(\cO_X) \in \perf(X\times X)$ of the diagonal $\Delta \subset X\times X$. 

Let $\cE$ be a perfect complex of $\cO_{X\times Y}$-modules, \ie a morphism in $\Var(\bbC)$ from $X$ to $Y$. As explained by Caldararu-Willerton \cite[\S3.2 and Appendix]{Caldararu}, its right adjoint in the $2$-category $\Var(\bbC)$ is given by $\Sigma_X \circ \cE^\vee$, where $\cE^\vee \in \perf(Y\times X)$ is the dual of $\cE$ and $\Sigma_X \in \perf(X\times X)$ is the Serre kernel $R\Delta_\ast(\omega_X[d_X])$. Since the composition of Fourier-Mukai functors corresponds to the convolution of kernels (see Huybrechts \cite[Proposition 5.10]{huybrechts}), and $R\Delta_{\ast}(\omega_X)$ is mapped to $p^\ast\omega_X$ (via pull-back push-forward), we conclude that $\Sigma_X \circ \cE^\vee \simeq \cE_R:=\cE^\vee \otimes^L p^*\omega_X[d_X]$. In particular, we have a well-defined unit morphism $\Gamma: R\Delta_\ast(\cO_X) \to \cE_R \circ \cE$ in $\perf(X \times X)$. Now, recall from Caldararu-Willerton \cite[\S1.2]{Caldararu} that one has a well-defined $2$-functor
$$ \perf(-): \Var(\bbC) \too \Cat$$
with values in the $2$-category of categories. A $\bbC$-scheme $X$ is mapped to the derived category $\perf(X)$ of perfect complexes, a kernel $\cE \in \perf(X\times Y)$ to the Fourier-Mukai functor $\Phi_\cE:\perf(X) \to \perf(Y)$, and a morphism $\mu: \cE\Rightarrow \cE'$ in $\perf(X\times Y)$ to a natural transformation $\Phi_{\mu}:\Phi_\cE \Rightarrow \Phi_{\cE'}$. In particular, the above adjunction $(\cE,\cE_R)$ in $\Var(\bbC)$ gives rise to the classical adjunction of categories
$$
\xymatrix{
\perf(Y) \ar@<1ex>[d]^-{\Phi_{\cE_R}} \\
\perf(X) \ar@<1ex>[u]^-{\Phi_\cE}\,,
}
$$
with unit morphism $\gamma=\Phi_\Gamma:\Id \Rightarrow \Phi_{\cE_R} \circ \Phi_\cE$. Now, choose a representative $\overline{\Gamma}$ (\ie a morphism of complexes of $\cO_{X\times X}$-modules) for the unit morphism $\Gamma: R\Delta_\ast(\cO_X) \to \cE_R \circ \cE$. As proved by To{\"e}n \cite[\S8.3]{Toen}, the composition law in $\Hmo(\bbC)$ corresponds under the above bijection \eqref{eq:bijection-kernels1} to the convolution of kernels. Hence, $\overline{\Gamma}$ gives naturally rise to a morphism
$$ \Phi_{\overline{\Gamma}}^\dg: \Id= \Phi^\dg_{R\Delta_\ast(\cO_X)} \Rightarrow \Phi^\dg_{\cE_R \circ \cE} = \Phi^\dg_{\cE_R} \circ \Phi^\dg_\cE$$
between dg functors. By precomposing all this data with the inclusion $i_X^\dg: \cT_X^\dg \to \perf^\dg(X)$ one then obtains a well-defined morphism $\Phi_{\overline{\Gamma}}^\dg \circ i_X^\dg$ from the dg functor $i_X^\dg:\cT_X^\dg \to \perf^\dg(X)$ to the composed dg functor 
\begin{equation}\label{eq:composed-dg}
 \cT_X^\dg \stackrel{i_X^\dg}{\too} \perf^\dg(X) \stackrel{\Phi_\cE^\dg}{\too} \perf^\dg(Y) \stackrel{\Phi^\dg_{\cE_R}}{\too} \perf^\dg(X)\,.
 \end{equation}
We now claim that this data satisfies the conditions of the general Lemma~\ref{lem:quasi-bimodules} (with $\nu=\Phi_{\overline{\Gamma}}^\dg \circ i_X^\dg$). This follows from the equalities
\begin{eqnarray*}
\dgHo(\Phi_{\cE_R}^\dg) =\Phi_{\cE_R} & \dgHo(\Phi_\cE^\dg)=\Phi_\cE & \dgHo(\Phi^\dg_{\overline{\Gamma}}) = \Phi_\Gamma=\gamma
\end{eqnarray*}
and from the fact that the unit morphism $\eta:\Id \Rightarrow \Phi_{\cE_R} \circ \Phi_\cE$ of the adjunction $(\Phi_\cE,\Phi_{\cE_R})$ is an isomorphism when evaluated at any object of $\cT_X \subset \perf(X)$. Consequently, Lemma~\ref{lem:quasi-bimodules} furnish us a quasi-isomorphism
\begin{equation}\label{eq:quasi}
{}_{(\Phi_{\overline{\Gamma}}^\dg \circ i_X^\dg)}B : {}_{i_X^\dg}B \Rightarrow {}_{\eqref{eq:composed-dg}}B
\end{equation}
between two bimodules which belong to $\rep(\cT_X^\dg,\perf^\dg(X))$. Making use of the above bijection \eqref{eq:bij}, one then concludes that both sides of \eqref{eq:quasi} are the same morphism in the homotopy category $\Hmo(\bbC)$ or equivalently that the above diagram \eqref{eq:com2} is commutative. This achieves the proof.
\end{proof}
We now have all the ingredients needed  for the conclusion of the proof of Theorem~\ref{thm:main1}. By assumption, Kuznetsov's conjecture \ref{conj:Kuznetsov1} holds, Hence, item (ii) of Proposition~\ref{prop:FMequalsdg} furnish us a well-defined dg functor $\Phi_\cE^\dg:\perf^\dg(X) \to \perf^\dg(Y)$. By applying to it the Jacobian functor ${\bf J}(-)$ one then obtains a morphism of complex abelian varieties up to isogeny
\begin{equation}\label{eq:morph-isogeny}
{\bf J}(\Phi_\cE^\dg): {\bf J}(\perf^\dg(X)) \too {\bf J}(\perf^\dg(Y))\,.
\end{equation}
Since by assumption the bilinear pairings \eqref{eq:pairings1} (associated to $X$ and $Y$) are non-degenerate, \eqref{eq:morph-isogeny} identifies with a well-defined morphism $
\tau:\prod_{i=0}^{d_X -1} J_a^i(X) \to \prod_{i=0}^{d_Y-1} J_a^i(Y)$ in $\Ab(\bbC)_\bbQ$. This proves item (i).

Let us now prove item (ii). 
Note first that thanks to Lemma~\ref{lem:projection}, $\cT_X^\dg$ is a direct factor of $\perf^\dg(X)$ and hence an object of the category $\NChow(\bbC)_\bbQ$ of noncommutative Chow motives. Hence, by applying the Jacobian functor ${\bf J}(-)$ to \eqref{eq:com2} one obtains the following commutative diagram of complex abelian varieties up to isogeny
\begin{equation}\label{eq:diag4}
\xymatrix{
{\bf J}(\perf^\dg(X)) \ar[r]^-{{\bf J}(\Phi_\cE^\dg)} & {\bf J}(\perf^\dg(Y)) \ar[r]^-{{\bf J}(\Phi_{\cE_R}^\dg)} & {\bf J}(\perf^\dg(X)) \\
{\bf J}(\cT_X^\dg) \ar[u]^-{{\bf J}(i_X^\dg)} \ar@{=}[rr] && {\bf J}(\cT_X^\dg) \ar[u]_-{{\bf J}(i_X^\dg)}\,.
}
\end{equation}
Since by assumption ${\bf J}(\cT_X^{\perp,\dg})=0$, Lemma~\ref{lem:key3} below implies that ${\bf J}(i_X^\dg)$ is
an isomorphism. Using the commutativity of diagram \eqref{eq:diag4} one then concludes that ${\bf J}(\Phi_{\cE_R}^\dg)$ is a
retraction of ${\bf J}(\Phi_\cE^\dg)\simeq \tau$ or equivalently that $\tau$ is split injective. This concludes the proof of item (ii).
\begin{lemma}\label{lem:key3}
The morphism of complex abelian varieties up to isogeny
$$ {\bf J}(i_X^\dg): {\bf J}(\cT_X^\dg) \too {\bf J}(\perf^\dg(X))$$
is an isomorphism if and only if ${\bf J}(\cT_X^{\perp,\dg})=0$.
\end{lemma}
\begin{proof}
Since by hypothesis one has a semi-orthogonal decomposition $\perf(X) = \langle \cT_X, \cT_X^\perp \rangle$, the inclusions of dg categories $i_X^\dg:\cT_X^\dg \hookrightarrow \perf^\dg(X)$ and $i_X^{\perp,\dg}: \cT_X^{\perp,\dg} \hookrightarrow \perf^\dg(X)$ give rise to an isomorphism $\cT_X^\dg\oplus \cT_X^{\perp,\dg} \stackrel{\simeq}{\to} \perf^\dg(X)$ in $\Hmo_0(\bbC)$ (and hence in $\NChow(\bbC)_\bbQ$); see \cite[Thm.~6.3]{IMRN}. Hence, by applying to it the additive Jacobian functor ${\bf J}(-)$ one then obtains the following isomorphism 
$$[{\bf J}(i_X^\dg),\,\,{\bf J}(i_X^{\perp,\dg})]: {\bf J}(\cT_X^\dg) \oplus {\bf J}(\cT_X^{\perp,\dg}) \stackrel{\simeq}{\too} {\bf J}(\perf^\dg(X))$$
in $\Ab(\bbC)_\bbQ$. This isomorphism clearly implies our claim and so the proof is finished.
\end{proof}
Let us now prove item (iii). From the above arguments and constructions, one has a commutative diagram
$$
\xymatrix{
\perf^\dg(X) \ar[r]^-{\Phi_\cE^\dg} & \perf^\dg(Y) \\
\cT_X^\dg \ar[u]^-{i_X^\dg} \ar[r]_-\simeq & \cT_Y^\dg \ar[u]_-{i_Y^\dg}
}
$$
in the homotopy category $\Hmo(\bbC)$. Hence, by applying to it the Jacobian functor ${\bf J}(-)$ one then obtains the following commutative diagram of complex abelian varieties up to isogeny
\begin{equation}\label{eq:diagram-final}
\xymatrix{
{\bf J}(\perf^\dg(X)) \ar[r]^-{{\bf J}(\Phi_\cE^\dg)} & {\bf J}(\perf^\dg(Y)) \\
{\bf J}(\cT_X^\dg) \ar[u]^-{{\bf J}(i_X^\dg)} \ar[r]_-{\simeq} & {\bf J}(\cT_Y^\dg) \ar[u]_-{{\bf J}(i_Y^\dg)}\,.
}
\end{equation}
Since by assumption ${\bf J}(\cT_Y^{\perp,\dg})=0$, one concludes from Lemma~\ref{lem:key3} that the vertical morphisms in \eqref{eq:diagram-final} are isomorphisms. This implies that ${\bf J}(\Phi_\cE^\dg)$ is an isomorphism and so the proof of Theorem~\ref{thm:main1} is finished.
\section{Polarizations}\label{sect:ppol}
In this section we introduce the notion of {\em verepresentability}; see Definition~\ref{def:strong-repre}.
\subsection{Algebraic and rational representability}\label{sub:rational}
Given a $\bbC$-scheme $X$ of dimension $d_X$, let $CH^i_{\bbQ}(X)$ (resp. $CH^i_{\bbZ}(X)$) denote the rational (resp. integral) Chow group of cycles of codimension $i$. We will write $A^i_{\bbQ}(X)$ (resp. $A^i_{\bbZ}(X)$) for the subgroup of $CH^i_{\bbQ}(X)$
(resp. of $CH_{\bbZ}^i(X)$) of algebraically trivial cycles. Let $A$ be a complex abelian variety. A group homomorphism $g: A^i_{\bbZ}(X) \to A$ is called 
a \it regular map \rm if for every $\bbC$-scheme $T$ and for every algebraic map
$f: T \to A^i_{\bbZ}(X)$, the composite $g \circ f: T \to A$ is a morphism of $\bbC$-schemes. 
\begin{definition}{(see Beauvile \cite[Definition 3.2.3]{beauvilleprym})}
A complex abelian variety $A$ is called the \it algebraic representative \rm of $A^i_{\bbZ}(X)$
if there exists a universal regular map $G:A^i_{\bbZ} (X) \to A$, \ie if for every regular map $g: A^i_{\bbZ}(X) \to B$,
there is a unique morphism of abelian varieties $u : A \to B$ such that $u \circ G = g$.
In this case one says that $A^i_{\bbZ}(X)$ is {\em algebraically representable}. As explained by Beauville in {\em loc. cit.}, whenever it exists the algebraic representative is unique up to isomorphism.
The standard examples of algebraic representatives are the Picard variety $\mathrm{Pic}^0(X)$ (when $i=1$) and the Albanese variety
$\mathrm{Alb}(X)$ (when $i=d_X$).
\end{definition}
\begin{definition}{(see Vial \cite[Corollary 3.6]{Vial})}
A $\bbC$-scheme $X$ is called \it rationally representable \rm if for every integer $i$ there exists a curve $\Gamma$ and a surjective algebraic map $A^1_{\bbQ}(\Gamma) \twoheadrightarrow A^{i+1}_{\bbQ}(X)$. In this case the Abel--Jacobi maps $AJ^i : A^{i+1}_{\bbZ}(X) \to J^i_a(X)$ give rise
to isomorphisms $A^{i+1}_{\bbQ}(X) \simeq J^i_a(X)_{\bbQ}$ for every $0 \leq i \leq n-1$.
\end{definition}
\subsection{Principal and incidence polarizations}\label{sub:principal}
From now on, and until the end of \S\ref{sect:ppol}, we will assume that $X$ is of odd dimension $d_X=2n+1$.
\begin{definition}{(see Beauville \cite[\S 3.4]{beauvilleprym})}\label{def:I(z)}
Let $T$ be a $\bbC$-scheme.
\begin{itemize}
\item[(i)] The {\em (rational) divisorial self correspondences} of $T$ are defined as $\mathrm{Corr}_{\bbQ}(T):=
\mathrm{Corr}(T)\otimes \bbQ$, where $\mathrm{Corr}(T) := 
\mathrm{Pic}(T \times T)/(\mathrm{Pic}(T) \boxtimes \mathrm{Pic}(T))$, and $\mathrm{Pic}$ stands for the Picard group.
\item[(ii)] Let $z$ be a cycle in $CH^{n+1}_\bbQ(T \times X)$ and $p,q: X \times T \times T \to T\times T$ and $r: X\times T \times T \to T \times T$ the projection morphisms.
The {\em incidence correspondence $I(z)$ associated to $z$} is the equivalence class of the cycle
$ R r_\ast (p^\ast(z)\cdot q^\ast(z)) \in CH^1_\bbQ(T\times T)$ considered as an element in $\mathrm{Corr}_{\bbQ}(T)$. 
\end{itemize}
\end{definition}
\begin{definition}{(see Birkenhake-Lange \cite[\S4.1]{birlange} and Mumford \cite[\S 8]{mumford-book})}\label{def:polarization}
Let $A$ be a complex abelian variety.
\begin{itemize}
\item[(i)] A {\em polarization} of $A$ is the first Chern character $\Theta_A:=c_1(L)$ of a positive definite line bundle.
A morphism $f:(A,L) \to (B,M)$ of polarized abelian varieties is a morphism of complex tori such that $f^\ast c_1(M)=c_1(L)$.
As explained in {\em loc. cit.}, $\Theta_A$ can be equivalently defined as a divisional rational
self correspondence $\theta_A \in \mathrm{Corr}_\bbQ(A)$ corresponding to an isogeny between $A$ and its dual $\hat{A}$; consult also Beauville
\cite[\S0.2 and 3.4]{beauvilleprym}.

\item[(ii)] A polarization $\Theta_A = c_1(L)$ (corresponding to a $\theta_A$ in $\mathrm{Corr}_\bbQ(A)$)
is called {\em principal} if the line bundle $L$ is of type $(1,\ldots, 1)$.
A {\em principally polarized abelian variety} consists of a pair $(A,\theta_A)$, where $A$ is a complex abelian variety
and $\theta_A\in \mathrm{Corr}_\bbQ(A)$
is a principal polarization of $A$. In particular, $\theta_A$ is principal
if and only if it gives rise to an isomorphism between $A$ and $\hat{A}$; see Beauville
\cite[\S0.2]{beauvilleprym}. Every polarization is obtained from a principal  one via an isogeny; see
\cite[Proposition 4.1.2]{birlange}. Consequently, a principal polarization is unique up to isomorphisms of $A$.

\item[(iii)] A morphism $f: A \to B$ of complex tori is a {\em morphism
of principally polarized abelian varieties} if, given a principal polarization $\Theta_B$, the pull back $f^*\Theta_B$ is
a principal polarization for $A$. Equivalently, the divisorial self correspondence $f^*\theta_B$ gives rise to an isomorphism $A \simeq \hat{A}$. Indeed, $f^*\theta_B$ is a polarization on
$A$ which is principal if and only if it gives such an isomorphism. Hence, $f^*\theta_B$ equals $\theta_A$ up to
an isomorphism of $A$.
\end{itemize}
\end{definition}
Recall from \S\ref{sec:introduction} that $\Ab(\bbC)_\bbQ$ stands for the category of abelian varieties up to isogeny.
\begin{lemma}\label{lemma:from-uptoiso-to-iso}
Let $(A,L)$ and $(B,M)$ be two complex polarized
abelian varieties, and $A_\bbQ$ and $B_\bbQ$ the classes of $A$ and $B$ in $\Ab(\bbC)_\bbQ$. Given a morphism $\lambda_{\bbQ}:A_{\bbQ} \to B_{\bbQ}$, there exists a morphism of complex abelian varieties $\lambda: A \to B$ whose class in $\Ab(\bbC)_\bbQ$ is $\lambda_{\bbQ}$.
Moreover, whenever $\lambda_{\bbQ}$ is split injective, $\lambda: A \to B$ is an isogeny onto a polarized abelian
subvariety of $B$.
\end{lemma}
\begin{proof}
Without loss of generality one can assume that $\lambda_\bbQ$ is surjective. In fact, since $(B,M)$ is a polarized abelian variety, every algebraic subtorus of $B$ is an abelian variety for which the restriction of $M$ is a polarization; see Birkenhake-Lange \cite[Proposition~4.1.1]{birlange}. Now, note that $\lambda_\bbQ$ is an element of ${\mathrm{Hom}}_{\Ab(\bbC)}
(A,B) \otimes \bbQ$ and hence can be written as a finite sum $\sum_i f_i \otimes \frac{p_i}{q_i}$. By first choosing a representative $A$ of the class $A_\bbQ$ and then by applying the morphism $\lambda_\bbQ$ to $A$ we obtain an algebraic complex torus $B'$ as $\lambda(A)_\bbQ$. The isogeny class $B'_{\bbQ}$ of $B'$
is the same as the isogeny class of $\lambda_{\bbQ}(A_{\bbQ})=B_\bbQ$. In particular, $B'$ is an abelian variety isogenous to $B$. As a consequence, one can take for $\lambda:A \to B$ the  morphism of complex tori obtained by composing $\lambda_{\bbQ}$ with
the isogeny $B' \to B$. Finally, the last claim follows from the above construction of $\lambda$. Whenever $\lambda_\bbQ$ is split injective, the kernel ${\mathrm{ker}}(\lambda)$ is torsion.
\end{proof}
\begin{definition}{(see Beauville \cite[D{\'e}finition~3.4.2]{beauvilleprym})}
Assume that $A^{n+1}_\bbZ(X)$ admits an algebraic representative $G: A^{n+1}_\bbZ(X) \to A$. In this case, a principal polarization $(A,\theta_A)$ is called the {\em incidence polarization with respect to $X$} if for all algebraic maps $f:T \to A^{n+1}_\bbZ(X)$ defined by a cycle $z \in CH_\bbQ^{n+1}(T\times X)$ the equality $(G \circ f)^\ast(\theta_A)=(-1)^{n+1} I(z)$ holds.
\end{definition}
\subsection{Verepresentability}\label{sub:verepresentable}
\begin{definition}\label{def:strong-repre}
A $\bbC$-scheme $X$ of odd dimension $d_X=2n+1$ is called {\em verepresentable} if $A^{i+1}_{\bbZ}(X)=0$ for $i \neq n$, if $A^{n+1}_{\bbZ}(X)$ admits an algebraic representative $A$ carrying an incidence polarization, and if the Abel-Jacobi map $AJ^n: A^{n+1}_{\bbZ}(X)
\twoheadrightarrow J^n_a(X)$ gives rise to an isomorphism $A^{n+1}_{\bbQ}(X) \simeq J^n_a(X)_{\bbQ}$. In this case, $A$ equals $J(X):=J^n_a(X)$
as a principally polarized abelian variety.
\end{definition}

Here is a list of examples of verepresentable $\bbC$-schemes:
\begin{example}{(Curves)}\label{ex:very1} 
Every curve $C$ is verepresentable. This follows automatically from the fact that there is a single group $A^1_{\bbZ}(X)$ of algebraically trivial
cycles and that $A^1_{\bbZ}(C)={\mathrm{Pic}}^0(C)=J(C)$.
\end{example}

\begin{lemma}\label{lem-vere-3fold-is-ratrep}
Let $X$ be a verepresentable threefold. Then, the bilinear pairings \eqref{eq:pairings1}
are non-degenerate.
\end{lemma}
\begin{proof}
By definition, one has $A^3_{\bbQ}(X)=0$. Hence, $A^3_{\bbQ}(X)$ is rationally representable, \ie there exists a curve
$\Gamma$ and a surjective algebraic morphism $A^1_{\bbQ}(\Gamma) \twoheadrightarrow A^3_{\bbQ}(X)$. 
Gorchinskiy-Guletskii \cite[Thm.~5.1]{gorch-gul-motives-and-repr} proved that
the rational representability of $A^3_{\bbQ}(X)$
is enough to describe a Chow-K\"unneth decomposition of the Chow motive of $M(X)_{\bbQ}$.
As explained by Vial in \cite[Thm.~4]{Vial}, this is equivalent to the
rational representability of all the $A^i_{\bbQ}(X)$. This implies that $X$ satisfies
the standard conjecture of Lefschetz type (see Vial \cite[Thm.~4.10]{Vial}) and hence that the bilinear pairings \eqref{eq:pairings1} are non-degenerate.
\end{proof}
Thanks to the work of Gorchinskiy-Guletskii \cite{gorch-gul-motives-and-repr}, whenever $X$ is a Fano
threefold, a conic bundle over a rational surface, or a del Pezzo fibration over $\bbP^1$,
$X$ is rationally representable, the bilinear pairings \eqref{eq:pairings1} are non-degenerate
(see the proof of Lemma \ref{lem-vere-3fold-is-ratrep}), and $J^i_a(X)$ is trivial for $i \neq 1$
(we have $\mathrm{Pic}^0(X)=0$). In these cases, the verepresentability of $X$ depends only on the existence of an incidence polarization.
Hence, here is a list of verepresentable threefolds:
\begin{example}{(Threefolds)}\label{ex:very2}
In the following cases, the canonical bundle of $X$ is always antiample.
\begin{itemize}
\item[(i)] {\bf Trivial Jacobian:} assume that $X$ is $\bbP^3$, a quadric threefold, a Fano of index 2 and degree 5,
or a Fano of index 1 and degree 22.
In all these cases the verepresentability follows easily from $J(X)=0$. In fact, $h^{1,2}(X)=0$; see Iskovskikh-Prokhorov \cite[\S 12.2]{isko-prok-fano}.

\item[(ii)] {\bf Fano threefolds of index 2:} thanks to the work of Clemens-Griffiths, Donagi, Reid, Tihomirov, and Voisin,
one can assume that $X$ is a cubic threefold~\cite{clemensgriffiths},
a quartic double solid~\cite{tihoquarticsolid,voisin-quartic-double}, the intersection of two quadrics in $\bbP^5$~\cite{donagi-torelli,reid:thesis}, or a Fano of index 2 and degree 5 (see item (i)). In all these cases there is an incidence polarization.

\item[(iii)] {\bf Fano threefolds of index 1:} thanks to the work of Beauville,
Bloch-Murre, Ceresa-Verra, Debarre, Iliev, Iliev-Manivel, Iliev-Markushevich, Iskovskikh-Prokhorov, Logach\"ev, Mukai
one can assume that $X$ is a general sextic double solid~\cite{ceresaverra}, a quartic in $\bbP^4$~\cite{blochmurreFano}, the intersection of a cubic and a quadric in $\bbP^5$~ \cite{blochmurreFano},
the intersection of three quadrics in $\bbP^6$~\cite{beauvilleprym,blochmurreFano}, a Fano of index 1 and degree 10~\cite{iliev10,logachevV10} or degree 14~\cite{iliev-marku}, a general Fano of index 1
 and degree 12~\cite{ilievmarkuV12}, degree 16~\cite{ilievV16,mukaigranmisto}, or degree 18 \cite{ilievamanovella,isko-prok-fano},
or a Fano of index 1 and degree 22 (see item (i)). In all these cases there is an incidence polarization.

\item[(iv)] {\bf Conic bundles:} thanks to the work of Beauville and Beltrametti (see \cite{beauvilleprym,beltrachow}),
one can assume that $X \to S$ is a standard conic bundle over a rational surface.
In all these cases there is an incidence polarization.
\item[(v)] {\bf del Pezzo fibrations:} thanks to the work of Kanev (see \cite{kanevdp1,kanevdp2}), one can assume that $X \to \bbP^1$
is a del Pezzo fibration of degree $d$ and that $2 \leq d \leq 5$. In these cases there is an incidence polarization.
\end{itemize}
For Fano threefolds of higher Picard rank we invite the
interested reader to consult the exhaustive treatment of Iskovskikh-Prokhorov \cite{isko-prok-fano}. Notice finally that there are still some threefolds of negative Kodaira
dimension (\eg\ Fanos of index 2 and degree 1, or del Pezzo fibrations over $\bbP^1$ of degree one)
for which verepresentability is not known to the authors.
\end{example}

\begin{example}{(Higher dimensions)}\label{ex:very3}
When $d_X\geq 5$ very few cases of verepresentable $\bbC$-schemes are known.
\begin{itemize}

\item[(i)] If $X$ is the intersection of two even dimensional quadrics, then $X$ is verepresentable. Thanks to
\cite[Thm. 1.5]{marcello-goncalo-chowgroups} and the work of Reid and Donagi \cite{reid:thesis,donagi-torelli},
the only nontrivial Jacobian is the intermediate one, and it carries an incidence polarization.

\item[(i')] If $X$ is the intersection of three odd dimensional quadrics, then $X$ is also verepresentable.
Once again, thanks to \cite[Thm. 1.5]{marcello-goncalo-chowgroups} and the work of Beauville \cite[\S 6]{beauvilleprym},
the only nontrivial Jacobian is the intermediate one, and it carries an incidence polarization.

\item[(ii)] If $X$ is an even dimensional quadric fibration over $\bbP^1$, then $X$
is verepresentable. This follows from the combination of Vial's motivic description \cite[\S 4]{vial-fibrations} with Reid's work on the intermediate Jacobian \cite{reid:thesis}; see also Donagi \cite{donagi-torelli}.

\item[(ii')] If $X$ is a odd dimensional quadric fibration over $\bbP^2$,
then $X$ is verepresentable. Thanks to the work of Beauville \cite[\S 4]{beauvilleprym},
the only nontrivial Jacobian is the intermediate one, and it carries an incidence polarization.
It is natural to expect that $\bbP^2$ can be replaced by any rational surface, as in the 3-dimensional
case; see Beltrametti \cite{beltrachow}.
\end{itemize}
Finally, note that Conjecture \ref{conj:grass-pfaff} would provide us $3$ more 5-dimensional examples.
\end{example}
\begin{remark}
Thanks to homological projective duality, Examples (i)-(ii) (and (i')-(ii'))
can be considered to be dual; see Kuznetsov \cite[\S 5]{kuznetquadrics}.
\end{remark}

\section{Proof of Theorem \ref{thm:main2}}
The $\bbC$-schemes $X$ and $Y$ satisfy all the assumptions of items (i)-(ii) of Theorem~\ref{thm:main1} and are moreover verepresentable.
In particular, $X$ (resp. $Y$) is irreducible of odd dimension $d_X:=2n+1$ (resp. $d_Y:=2m+1$). Moreover, there is a single
non-trivial algebraic Jacobian $J(X):=J^n_a(X)$ (resp. $J(Y):=J^m_a(Y)$), which via a universal regular map
\begin{eqnarray*}
G_X: A_\bbZ^{n+1}(X) \twoheadrightarrow J(X) && \text{(resp.}\,\,\, G_Y: A^{m+1}_\bbZ(Y) \twoheadrightarrow J(Y)\, \text{)},
\end{eqnarray*}
is the algebraic representative of $A^{n+1}_\bbZ(X)$ (resp. of $A^{m+1}_\bbZ(Y)$); see \S\ref{sub:rational}.
A proof of the surjectivity of $G_X$ (resp. $G_Y$) can be found in Beauville's work \cite[Remark 3.2.4(ii)]{beauvilleprym}. Furthermore, we have induced isomorphisms
\begin{eqnarray*}
AJ^n_\bbQ:A_\bbQ^{n+1}(X) \stackrel{\simeq}{\to} J(X)_\bbQ && AJ^m_\bbQ:A_\bbQ^{m+1}(X) \stackrel{\simeq}{\to} J(Y)_\bbQ\,,
\end{eqnarray*}
where $J(X)_\bbQ$ (resp. $J(Y)_\bbQ$) stands for $J(X)$ (resp. $J(Y)$) considered as an abelian variety {\em up to isogeny}.

Now, recall from the proof of Theorem~\ref{thm:main1} that there exists a perfect complex $\cE\in \perf(X\times Y)$
such that the split injective morphism $\tau:J(X)_\bbQ \to J(Y)_\bbQ$ is obtained by applying the Jacobian
functor ${\bf J}(-)$ to the Fourier-Mukai dg functor $\Phi_\cE^\dg:\perf^\dg(X) \to \perf^\dg(Y)$. Consider the following homomorphism
\begin{eqnarray*}
e_{m+n+1}:CH^{n+1}_\bbQ(X) \to CH^{m+1}_\bbQ(Y) && z \mapsto q_\ast(p^\ast(z)\cdot ch(\cE)_{m+n+1})\,,
\end{eqnarray*}
where $ch(\cE)_{m+n+1} \in CH_\bbQ^{m+n+1}(X\times Y)$ is the $(m+n+1)^{\mathrm{th}}$-component of the Chern character $\ch(\cE)$ of $\cE$, and $p:X\times Y \to X$ and $q:X\times Y \to Y$ are the projection morphisms.
Recall that algebraic equivalence is an \it adequate \rm equivalence relation on cycles; see Andr{\'e} \cite[D\'ef 3.1.1.1, \S 3.2.1]{Andre}.
This, together with the fact that $q: X \times Y \to Y$ is equidimensional, implies that
one obtains a map $\overline{\tau}:A^{n+1}_{\bbQ}(X) \to A^{m+1}_{\bbQ}(Y)$, which is
still the correspondence given by $\ch(\cE)_{m+n+1}$; see Andr\'e \cite[\S 3.1.2]{Andre}.
\begin{lemma}\label{lem:key1}
One has the following commutative diagram
$$
\xymatrix{
A^{n+1}_\bbQ(X) \ar[d]^-{\simeq}_-{AJ^n_\bbQ} \ar[r]^-{\overline{\tau}} & A^{m+1}_\bbQ(Y) \ar[d]_-{\simeq}^-{AJ^m_\bbQ} \\
J(X)_\bbQ \ar[r]_-\tau & J(Y)_\bbQ\,.
}
$$
\end{lemma}
\begin{proof}
Let us denote by $\Chow^\ast(\bbC)_\bbQ$ (resp. by $\Num^\ast(\bbC)_\bbQ$) the category of Chow (resp. numerical)
motives where the morphisms are the graded correspondences. As proved in \cite[Thm.~1.9]{Semi},
the category $\Num^\ast(\bbC)_\bbQ$ (denoted by $\Num(\bbC)_\bbQ/_{\!\!-\otimes \bbQ(1)}$ in {\em loc. cit.})
is abelian semi-simple. Recall from \cite[\S4]{MT} that one has a well-defined inclusion $\Ab(\bbC)_\bbQ \subset \Num^\ast(\bbC)_\bbQ$
of abelian semi-simple categories. Now, consider the following graded correspondence
$$ e:= ch(\cE) \cdot p^\ast \mathrm{Td}(X) \in \oplus_i CH^i_\bbQ(X\times Y)\,,$$
where $\mathrm{Td}(X)$ is the Todd class of $X$. As proved in \cite[\S4]{MT} (see also \cite[\S8]{CvsNC}), $J(X)_\bbQ$ (resp. $J(Y)_\bbQ$)
is the largest direct summand of $M_\bbQ(X) \in \Num^\ast(\bbC)_\bbQ$ (resp. of $M_\bbQ(Y)$) which belongs to $\Ab(\bbC)_\bbQ$.
Moreover, the morphism $\tau={\bf J}(\Phi_\cE^\dg):J(X)_\bbQ \to J(Y)_\bbQ$ is the largest direct summand of the graded correspondence $e$
(considered as a morphism in $\Num^\ast(\bbC)_\bbQ$ from $M_\bbQ(X)$ to $M_\bbQ(Y)$) which belongs to $\Ab(\bbC)_\bbQ$. The correspondence $e$ is a mixed cycle. Let us now show that the only
degree that contributes to $\tau$ is $e_{m+n+1}$; note that this automatically
achieves the proof. Since the relative dimension of the projection morphism $q:X\times Y \to Y$
is equal to the dimension of $X$, \ie $2n+1$, we have a well-defined {\em graded} homomorphism
$$ q_\ast: \mathrm{CH}^\ast_\bbQ(X\times Y) \too \mathrm{CH}^{\ast-2n-1}_\bbQ(Y)$$
and consequently we obtain the equality
$$(q_\ast(p^\ast(z)\cdot e))_{m+1}=q_\ast((p^\ast(z) \cdot ch(\cE) \cdot p^\ast\mathrm{Td}(X))_{m+2n+2}).$$
Moreover, since $p^*$ is a morphism of commutative rings, \ie it respects the commutative
intersection pairing (see \cite[A1-2, page 426]{hartshorne}), we also the equality
$$p^\ast(z) \cdot ch(\cE) \cdot p^\ast\mathrm{Td}(X) = p^*(z\cdot\Td(X))\cdot\ch(\cE)$$
Now, recall from Pappas \cite[\S1]{Pappas} that the Todd class $\mathrm{Td}(X)$
is given by $\sum_{i \geq 0} \frac{\mathfrak{D}_i}{T_i}$. Here, $\mathfrak{D}_i$ is a polynomial with integral coefficients in
the Chern classes and $T_i$ is the product $\prod_p p ^{[\frac{i}{p-1}]}$ taken over all the prime numbers; the symbol $[-]$ stands
for the integral part. Note that $\mathfrak{D}_0=1$. The algebraic equivalence relation on cycles is an \it adequate \rm 
equivalence relation; see Andr{\'e} \cite[\S 3.2.1]{Andre}. This means, in
particular, that by intersecting the algebraically trivial cycle $z$ with any other cycle one obtains an algebraically trivial cycle; see Andr{\'e} \cite[Definition 3.1.1.1(3)]{Andre}. By assumption, $A^i_{\bbZ}(X)=0$ for all $i \neq n+1$.
Hence, in the following intersection product
$$z \cdot \Td(X)= z \cdot \sum_{i \geq 0} \frac{\mathfrak{D}_i}{T_i}=
\sum_{i \geq 0} (z \cdot \frac{\mathfrak{D}_i}{T_i})\,,$$
all the components of degree $\neq n+1$ are trivial. Since $z$ has degree $n+1$, $z \cdot \Td(X) = z \cdot \frac{\mathfrak{D_0}}{T_0} = z$,
and therefore $p^\ast(z \cdot \Td(X))=p^\ast(z)$ is a purely $n+1$-codimensional cycle. We obtain in this way the following equality
$$(q_\ast(p^\ast(z)\cdot e))_{m+1} = q_\ast(p^*(z) \cdot \ch(\cE)_{m+n+1})$$
which achieves the proof.
\end{proof}

Recall now from the proof of Theorem~\ref{thm:main1} that the retraction $\sigma:J(Y)_\bbQ \to J(X)_\bbQ$ of $\tau$ is obtained
by applying the Jacobian functor ${\bf J}(-)$ to the Fourier-Mukai dg functor $\Phi^\dg_{\cE_R}:\perf^\dg(Y) \to \perf^\dg(X)$,
where $\cE_R:=\cE^\vee \otimes^L p^\ast \omega_X [d_X] \in \perf(Y \times X)$. Consider the following homomorphism
\begin{eqnarray}\label{eq:description1}
&\overline{\sigma_1}:A^{m+1}_\bbQ(Y) \too A_\bbQ^{n+1}(X) & w \mapsto p_\ast(q^\ast(w) \cdot (-1)^{m+n} ch(\cE)_{m+n+1})\,,
\end{eqnarray}
and let $\sigma_1: J(Y)_\bbQ \to J(X)_{\bbQ}$ be the unique morphism induced by $\overline{\sigma_1}$ via the Abel--Jacobi map $AJ_\bbQ$.

\begin{lemma}\label{lem:key2}
Whenever the canonical bundle of $X$ is ample or antiample, there is an automorphism
$\psi: X \stackrel{\simeq}{\to} X$ making the following diagram commute
$$
\xymatrix{
A^{m+1}_\bbQ(Y) \ar[d]^-{\simeq}_-{AJ^m_\bbQ} \ar[r]^-{\overline{\sigma_1}} & A^{n+1}_\bbQ(X) \ar[d]_-{\simeq}^-{AJ^n_\bbQ} \ar[r]^-{\psi^\ast} & A^{n+1}_\bbQ(X) \ar[d]_-{\simeq}^-{AJ^n_\bbQ} \\
J(Y)_\bbQ \ar[r]^-{\sigma_1} \ar@/_1pc/[rr]_-{\sigma} & J(X)_\bbQ \ar[r]^-{\psi^\ast} \ & J(X)_\bbQ\,.
}
$$
Whenever the canonical bundle of $X$ is trivial, we have $\sigma_1=\sigma$.
\end{lemma}
\begin{proof}
Let us assume first that the canonical bundle of $X$ is either ample or antiample.
Since the Abel--Jacobi map is natural (\ie it is a natural transformation between the functors $A^{n+1}(-)_\bbQ$ and $J(-)_\bbQ$),
we have the equality $\psi^\ast \circ AJ^n = AJ^n \circ \psi^*$. Hence, we need only to show that $\sigma=\sigma_1 \circ \psi^*$. Let us denote $\cE_!$ the perfect complex $\cE^{\vee} [d_X]$, and by $\Psi$ the autoequivalence of $\perf(X)$ given by the
tensorization with $\omega_X$.
Note that $\Phi_{\cE_R}= \Psi \circ \Phi_{\cE_!}$. Since by assumption $\omega_X$ is ample or antiample, the autoequivalence $\Psi$
induces an automorphism $\psi: X \stackrel{\simeq}{\to} X$; see Bondal-Orlov \cite[Thm.~2.5]{bondorl-reconst}. Similarly to the proof of Lemma \ref{lem:key1}, it suffices only to perform an explicit Chern class calculation on $\cE_!$ to obtain the searched commutative diagram. Consider the graded correspondence
$$e_!:=ch(\cE_!) \cdot p^\ast \mathrm{Td}(Y) \in \oplus_i CH_\bbQ^i(Y \times X)\,.$$
Once again, similarly to the proof of Lemma~\ref{lem:key1}, the morphism $\sigma={\bf J}(\Phi_{\cE_R}^\dg): J(Y)_\bbQ \to J(X)_\bbQ$ is obtained by composing $\psi^*$ with the largest direct summand of the graded correspondence $e_!$
(considered as a morphism in $\Num^\ast(\bbC)_\bbQ$ from $M_\bbQ(Y)$ to $M_\bbQ(X)$) which belongs to $\Ab(\bbC)_\bbQ$.
Hence, it suffices to show that \eqref{eq:description1} agrees with the homomorphism
\begin{eqnarray*}
A^{m+1}_\bbQ(Y) \too A^{n+1}_\bbQ(X) && w \mapsto (p_\ast(q^\ast(w) \cdot e_!))_{n+1}\,.
\end{eqnarray*}
The same calculations as in Lemma \ref{lem:key1} allow us to conclude that the cycle inducing $\sigma_1$ is $\ch_{m+n+1}(\cE_!)$. So the claim holds once we show the following equality
\begin{eqnarray}\label{eq:homo2}
 \ch_{m+n+1}(\cE_!)=\ch_{m+n+1}(\cE^\vee [d_X])=(-1)^{m+n} \ch_{m+n+1}(\cE)\,.
 \end{eqnarray}
Since $d_X=2n+1$ is odd we have $\ch(\cE^{\vee} [2n+1]) = -\ch(\cE^{\vee})$. Moreover, $\ch_i(\cE^{\vee})=(-1)^i\ch_i(\cE)$. This follows from the fact that the Chern polynomial
${\mathrm c}_t(\cE^{\vee})$ is obtained from ${\mathrm c}_t(\cE)$ by alternating
signs; see Hartshorne \cite[A.3]{hartshorne}. This implies that $\ch_i(\cE^{\vee})=(-1)^i\ch_i(\cE)$ and
consequently that
$$\ch_{m+n+1}(\cE^{\vee} [2n+1]) =-\ch_{m+n+1}(\cE^{\vee})=(-1)^{m+n}\ch_{m+n+1}(\cE).$$
The above equality \eqref{eq:homo2} follows now from the fact that $\cE^{\vee}[2n+1]= \cE_!$.

Finally, let us assume that the canonical bundle of $X$ is trivial. In this case we have $\cE_!=\cE_R$
and so the proof is similar to the one of Lemma~\ref{lem:key1}.
\end{proof}

It follows from Lemma \ref{lem:key2} that whenever the canonical bundle of $X$ is trivial,
we can consider $\psi= \id_X$ and $\sigma_1=\sigma$. In order to avoid repeating twice the same arguments, we will assume from now on that this choice has been made.

We now have all the ingredients needed for the conclusion of the proof of Theorem~\ref{thm:main2}.
Since by hypothesis $X$ (resp. $Y$) is verepresentable the non-trivial algebraic Jacobian $J(X)$ (resp. $J(Y)$)
is an abelian variety endowed with the incidence polarization $\theta_X$ (resp. $\theta_Y$) with respect to $X$ (resp. to $Y$).
In particular, this allows one to consider $J(X)$ and $J(Y)$ as canonical representative
for ${\bf J}(X) \simeq J(X)_{\bbQ}$ and ${\bf J}(Y) \simeq J(Y)_{\bbQ}$. Thanks to Lemma \ref{lemma:from-uptoiso-to-iso},
having a canonical choice of (principally polarized) abelian varieties representing $J(X)_{\bbQ}$ and $J(Y)_{\bbQ}$,
the split injective map $\tau: J(X)_{\bbQ} \to J(Y)_{\bbQ}$ can be seen as an isogeny $\tau: J(X) \to J(Y)$. Clearly, this isogeny
is described by the algebraic cycle $\ch_{m+n+1}(\cE)$; see Lemma \ref{lem:key1}.

On the other hand, Lemma \ref{lemma:from-uptoiso-to-iso} furnish us a morphism of algebraic tori $\sigma: J(Y) \to J(X)$. A similar argument furnish us
a morphism $\sigma_1: J(Y) \to J(X)$, described by the cycle $(-1)^{m+n} \ch_{m+n+1}(\cE)$, and also an isomorphism $\psi^\ast:J(X) \stackrel{\simeq}{\to} J(X)$.
Moreover, thanks to Lemma \ref{lem:key2}, the following equality holds $\sigma = \psi^\ast \circ \sigma_1$. Finally, note that since $\psi: X \stackrel{\simeq}{\to} X$ is an
isomorphism, $\psi^\ast$ respects the principal polarization.

Let us now show that the isogeny $\tau: J(X) \to J(Y)$, considered as a morphism
of abelian varieties, pulls-back the principal polarization to a principal polarization. In order to do it, we will make use of the incidence property of $\theta_Y$. By hypothesis, $J(Y)$ is the algebraic representative of $A^{m+1}_{\bbZ}(Y)$ and the principal polarization
$\theta_{Y}$ of $J(Y)$ is the incidence polarization with respect to $Y$.
If $f:T \to A^{m+1}_{\bbZ}(Y)$ is an algebraic map
defined by a cycle $z$ in $CH_{\bbQ}^{m+1}(T \times Y)$, then (as explained in \S\ref{sub:principal}) we have the following equality
\begin{equation}\label{eq:incid-prop}(G \circ f)^* \theta_Y =(-1)^{m+1} I(z),\end{equation}
of divisorial self correspondences on $T$; see Definition \ref{def:I(z)}(ii). Note that $f$ is not
necessarily surjective. Set $T=J(X)$ and consider the cycle $\ch_{m+n+1}(\cE)$ as giving rise to an algebraic map $\tilde{\tau}:J(X) \to A^{m+1}_{\bbZ}(Y)$:
points of the algebraic variety $J(X)$ can be considered as equivalence classes of cycles on $X$.
Hence, an element $\alpha \in J(X)$ can be thought as a $0$-dimensional cycle on the algebraic variety $J(X)$ and also as an equivalence class of a $n+1$-codimensional cycle on $X$. In order to construct an algebraic map $J(X) \to A^{m+1}_\bbZ(Y)$,
recall that $G_X: A^{n+1}_\bbZ(X) \twoheadrightarrow J(X)$ is surjective. Hence, choose a representative of $\alpha \in A^{n+1}_{\bbZ}(X)$ and then use the
cycle $\ch_{m+n+1}(\cE)$. Since the map $\tau$ is induced by such a correspondence, the choice of a representative becomes irrelevant as soon
as we consider the map $G_Y: A^{m+1}_\bbZ(Y) \twoheadrightarrow J(Y)$. In particular, we have $G_Y \circ \tilde{\tau}=\tau$. By applying the incidence property \eqref{eq:incid-prop} to the algebraic map $\tilde{\tau}$ we then obtain the following equality
\begin{equation}\label{eq:equality-final}
\tau^\ast \theta_Y =(G_Y \circ \tilde{\tau})^* \theta_Y= (-1)^{m+1} \mathfrak{i},
\end{equation}
where $\mathfrak{i}:=I(\ch_{m+n+1}(\cE))$ is the incidence correspondence of $\ch_{m+n+1}(\cE)$; see Definition \ref{def:I(z)}. 
The principal polarization $\theta_X$ allows to identify $J(X)$ with its dual; see Definition \ref{def:polarization}. In order to conclude the proof it remains only to show that $\mathfrak{i}$ is isomorphic, via $\psi^*$
and up to a sign, to the identity correspondence (here we
identify $J(X)$ with its dual). Thanks to Lemmas \ref{lem:key1}-\ref{lem:key2} and Definition \ref{def:I(z)},
we have $\sigma_1 \circ \tau= (-1)^{m+n}\mathfrak{i}$. Moreover, the following commutative diagram of algebraic tori holds:
$$\xymatrix{\id: J(X) \ar@/_1pc/[rr]_{(-1)^{m+n}\mathfrak{i}}\ar[r]^{\tau} & J(Y) \ar@/^1.5pc/[rr]^{\sigma}\ar[r]^{\sigma_1} & J(X) \ar[r]_{\psi^\ast} & J(X). }$$
As a consequence, $\psi^\ast \circ \mathfrak{i} = (-1)^{m+n} \id$. This achieves the first claim of Theorem~\ref{thm:main2}. The second claim follows automatically from the arguments above and from item (iii) of Theorem~\ref{thm:main1}. 

\section{Remaining proofs}\label{sec:proofs}
\subsection*{Proof of Proposition~\ref{prop:double-torelli}}
\begin{proof}
Since the (generalized) Torelli theorem holds in each one of the cases (i)-(iii), the implication $(\Leftarrow)$ follows from
Corollary~\ref{cor:Torelli}. Now, suppose that $X$ and $Y$ are isomorphic via an isomorphism $f^\ast:X\stackrel{\simeq}{\to} Y$.
Clearly, the inverse image functor gives rise to an equivalence $f^\ast: \perf(Y) \stackrel{\simeq}{\to} \perf(X)$ which is
of Fourier-Mukai type with kernel given by structure sheaf of the graph of $f$; see Huybrechts \cite[Example 5.4]{huybrechts}.
It remains then only to show that this equivalence restricts to an equivalence $\cT_X \simeq \cT_Y$.
Recall from \S\ref{sec:app-4} that when $X$
and $Y$ are elliptic curves (\ie item (ii) of Proposition \ref{prop:double-torelli} with $d_X=d_Y=1$,
or equivalently $r=1$, $n=2$ and $T=\mathrm{point}$ in the notations of \S\ref{sec:app-4}), we have $C_X=X$, $C_Y=Y$ and the equivalence $\perf(X)
= \perf(C_X) \simeq \perf(C_Y) = \perf(Y)$. In all the remaining cases, $X$ and $Y$ are Fano varieties with Picard number one and index $i$,
\ie $\omega_X\simeq \cO_X(-i)$ and $\omega_Y \simeq \cO_Y(-i)$. In particular $\mathrm{Pic}(X)\simeq \bbZ[\cO_X(1)]$ and $\mathrm{Pic}(Y)\simeq \bbZ[\cO_Y(1)]$. Thanks to the work of Kuznetsov
\cite{kuznet-v14,kuznetfanothreefolds,kuznetquadrics}, one has moreover the following semi-orthogonal decompositions
\begin{eqnarray*}
\perf(X) =\langle \cT_X, \langle\cO_X, \ldots, \cO_X(i-1)\rangle\rangle && \perf(Y) =\langle \cT_Y, \langle\cO_Y, \ldots, \cO_Y(i-1)\rangle\rangle\,.
\end{eqnarray*}
Hence, the isomorphism $f:X\stackrel{\simeq}{\to} Y$ gives rise to an isomorphism 
\begin{eqnarray*}
f^\ast: \bbZ[\cO_Y(1)]\simeq \mathrm{Pic}(Y) \stackrel{\simeq}{\to} \mathrm{Pic}(X) \simeq \bbZ[\cO_X(1)] && \cO_Y(j) \mapsto \cO_X(j)
\end{eqnarray*}
of abelian groups. Consequently, $\cT_X^\perp \simeq f^\ast \cT_Y^\perp$ and so the above equivalence $f^\ast: \perf(Y) \stackrel{\simeq}{\to} \perf(X)$ restricts to an equivalence $\cT_Y \simeq \cT_X$. This achieves the proof.
\end{proof}
\subsection*{Proof of Theorem~\ref{thm:conjectures}}
\begin{proof}
Recall from Debarre-Iliev-Manivel \cite[\S 4.1]{debarre-iliev-manivel} that when $d=2$, $Y_2$ is a quartic double solid and that
the generic $X_{10}$ is a linear section of a quadric hypersurface inside ${\mathrm{Gr}}(2,\bbC^5)$. Thanks to the work of Clemens, Iliev, Logach\"ev and Tihomirov
(see \cite{clemens4ic,iliev10,logachevV10,tihoquarticsolid}), $Y_2$ and $X_{10}$ have only have one non-trivial principally polarized
intermediate Jacobian carrying an incidence polarization. Hence, these schemes satisfy all the assumptions of Theorem \ref{thm:main2}; see Lemma \ref{lem-vere-3fold-is-ratrep} and Example \ref{ex:very2}.

Now, note that Conjecture \ref{conj:Kuznetsov2} implies that for every $X_{10}$ there exists at least one quartic double solid $Y_2$
such that the triangulated categories $\cT_{X_{10}}$ and $\cT_{Y_2}$ are equivalent. On the other hand, Conjecture \ref{conj:Kuznetsov1} implies that the composed functor
$$ \langle \cT_{X_{10}}, \langle \cE_{X_{10}}, \cO_{X_{10}} \rangle \rangle \to \cT_{X_{10}} \simeq \cT_{Y_2} \hookrightarrow \langle \cT_{Y_2}, \langle
\cO_{Y_2},\cO_{Y_2}(1) \rangle \rangle$$
is of Fourier-Mukai type. Note that since the canonical bundle of $X$ is antiample, all assumptions of Theorem \ref{thm:main2} are verified.
One concludes then that $J(X_{10}) \simeq J(Y_2)$ as principally polarized abelian varieties.
The key point now is that this isomorphism {\em does not exist} in full generality. As explained by Debarre-Iliev-Manivel
\cite[Corollary~5.4]{debarre-iliev-manivel}, varying $X_{10}$ in $\cM\cF^1_{10}$ gives rise to a $20$-dimensional family of intermediate Jacobians $J(X_{10})$; consult also \cite[Thm.~5.1]{debarre-iliev-manivel}. However, $\cM\cF^2_{2}$ is only 19-dimensional. Hence, we conclude that the correspondence $Z_2$ (whose existence is conjectured by Kuznetsov) can never be dominant
onto $\cM\cF^1_{10}$. This achieves the proof.
\end{proof}
\subsection*{Proof of Theorem~\ref{thm:blow-ups}}
\begin{proof}
Let $f:\widetilde{Y} \to Y$ be the blow-up of a $\bbC$-scheme $Y$ along a smooth center $D$.
As proved by Orlov \cite{orlovprojbund}, one has the following semi-orthogonal decomposition
\begin{equation}\label{eq:orthogonal}
\perf(\widetilde{Y})=\langle f^\ast \perf(Y), \langle\underbrace{\perf(D),\ldots, \perf(D)}_{(d-1)\text{-}\mathrm{factors}} \rangle \rangle \,,
\end{equation}
where $d$ is the codimension of $D$ in $Y$. In the particular cases $D:=\mathrm{pt}$ and $D:=C$, 
the above semi-orthogonal decomposition \eqref{eq:orthogonal} reduces to 
\begin{eqnarray*}
\perf(\widetilde{X_{\mathrm{pt}}})\simeq\langle f^\ast \perf(X), \perf(\mathrm{pt}) , \perf(\mathrm{pt}) \rangle && \perf(\widetilde{X_C}) = \langle f^\ast \perf(X), \perf(C)\rangle\,.
\end{eqnarray*}
Now, recall from Clemens-Griffiths \cite{clemensgriffiths} that in dimension $3$, $\widetilde{X_{\mathrm{pt}}}$ and $\widetilde{X_C}$
are verepresentable whenever $X$ is verepresentable. Let us now verify that 
the remaining conditions of Theorem~\ref{thm:main2} are satisfied. In the case where $D:=\mathrm{pt}$, take $\cT_X:=\perf(X)$.
Clearly, the equivalence $\perf(X) \simeq f^\ast \perf(X)$ gives rise to a composed functor $\Phi:\perf(X) \to
\perf(\widetilde{X_{\mathrm{pt}}})$ which is of Fourier-Mukai type with kernel given by structure sheaf of $f$.
Since $X$ is three-dimensional and verepresentable, the bilinear pairings \eqref{eq:pairings1} are non-degenerate for $X$ 
(see Lemma \ref{lem-vere-3fold-is-ratrep}), and this holds also for $\widetilde{X_{\mathrm{pt}}}$.
Finally, the canonical bundle of $X$ is either ample, antiample or trivial. As a consequence, all the conditions of Theorem~\ref{thm:main2}
are satisfied. Hence, one obtains an isomorphism $J(\widetilde{X_{\mathrm{pt}}})\simeq J(X)$ of principally polarized abelian varieties.

In the case $D:=C$, one needs to apply Theorem~\ref{thm:main2} twice. First take $\cT_X:=\perf(X)$. Clearly, the equivalence $\perf(X)
\simeq f^\ast \perf(X)$ gives rise to a functor $\Phi:\perf(X) \to \perf(\widetilde{X_{C}})$ which is of Fourier-Mukai type. Since $X$ is three-dimensional and
rationally representable, the bilinear pairings \eqref{eq:pairings1} are non-degenerate (see Vial \cite[\S 4]{Vial}),
and this holds also for $\widetilde{X_C}$.
Finally, the canonical bundle of $X$ is either ample, antiample or trivial.
As a consequence, all conditions of Theorem \ref{thm:main2} (except ${\bf J}(\cT_Y^{\perp,\dg})=0$) are satisfied. Hence, one obtains
a split injective morphism $J(X) \to J(\widetilde{X_C})$ of principally polarized abelian varieties. 
Now, take $\cT_X:=\perf(C)$. The inclusion of categories $\Phi:\perf(C) \hookrightarrow \perf(\widetilde{X_C})$ is 
of Fourier-Mukai type; see Orlov \cite[Assertion 4.2]{orlovprojbund}\cite{orlov-represent} or Huybrechts \cite[Proposition 11.16]{huybrechts}. Moreover, the bilinear pairings \eqref{eq:pairings1}
are non-degenerate for $C$, and $C$ is verepresentable. Finally the canonical bundle of $C$ is either ample, antiample
or trivial. As a consequence, all the conditions of Theorem~\ref{thm:main2} are satisfied. Hence, one obtains a split injective
morphism $J(C) \to J(\widetilde{X_C})$ of principally polarized abelian varieties. Now, recall that we have the following semi-orthogonal
decomposition $\perf(\widetilde{X_C})=\langle f^\ast \perf(X),\perf(C) \rangle$. By construction of the Jacobian functor we hence obtain
an isomorphism ${\bf J}(\perf^\dg(\widetilde{X_C}))\simeq {\bf J}(\perf^\dg(X)) \times {\bf J}(\perf^\dg(C))$ in $\Ab(\bbC)_\bbQ$.
As proved in Lemma \ref{lemma:from-uptoiso-to-iso}, this gives rise to an isogeny $J(\widetilde{X_C}) \simeq J(X) \times J(C)$ and therefore to an injective morphism $\alpha: J(X) \times J(C) \to J(\widetilde{X_C})$ of principally polarized
abelian varieties. In order to prove its surjectivity, consider its cokernel $A'$. By the previous considerations,
$A'$ is a principally polarized abelian variety isogenous to $0$. Hence, $A'=0$ as principally polarized abelian
variety and consequently $\alpha$ is surjective as well.
\end{proof}
\subsection*{Proof of Theorem~\ref{thm:morita-hyperbolic}}
\begin{proof}
The proof will consist on verifying all conditions of Theorem~\ref{thm:main1}.
Note first that thanks to Kuznetsov and others (see \cite[Thm.~7.1]{kuznetbasechange}\cite{kuznetquadrics} and \cite{auel-bernar-bologn}) the composed functor $\Phi: \perf(Q) \to \perf(S,\cC_0) \simeq \perf(S,\cC_0') \hookrightarrow \perf(Q')$ is of Fourier-Mukai type. As proved by Kuznetsov \cite{kuznetquadrics}, we have the following semi-orthogonal decompositions
$$\perf(Q) = \langle \perf(S,\cC_0), \langle \underbrace{\perf(S),\ldots, \perf(S)}_{n\text{-}\mathrm{factors}} \rangle\rangle$$
$$\perf(Q') = \langle \perf(S,\cC_0'), \langle \underbrace{\perf(S),\ldots, \perf(S)}_{n-2\text{-}\mathrm{factors}} \rangle\rangle\,.$$
Consequently, since by hypothesis ${\bf J}(\perf^\dg(S))=0$, all the assumptions of Theorem~\ref{thm:main1} are verified. This achieves the proof.
\end{proof}
\subsection*{Proof of Theorem~\ref{thm:hpd-quadr}}
\begin{proof}
Similarly to the proof of Theorem~\ref{thm:morita-hyperbolic}, we will verify all conditions of Theorem~\ref{thm:main1}. Note first that thanks to relative homological projective duality (see \cite{auel-bernar-bologn} and Kuznetsov \cite[\S 6.1]{kuznetsov:hpd})
the composed functor $ \Phi:\perf(X) \to \perf(S,\cC_0) \simeq \perf(S,\cC_0) \hookrightarrow \perf(Q)$ is of Fourier-Mukai type. By construction of the Jacobian functor ${\bf J}(-)$ the above semi-orthogonal decompositions \eqref{eq:mot-decomp-1}-\eqref{hpd_Fano} give rise to the isomorphisms
$$
{\bf J}(\perf^\dg(Q)) \simeq {\bf J}(\perf^\dg(S,\cC_0))\times \prod^{n}_{i=0} {\bf J} (\perf^\dg(S))
$$
$${\bf J}(\perf^\dg(X)) \simeq {\bf J}(\perf^\dg(S,\cC_0))\times \prod^{n-2r}_{i=0} {\bf J} (\perf^\dg(T))$$
in $\Ab(\bbC)_\bbQ$. Hence, it suffices to show that ${\bf J}(\perf^\dg(S))=0$. Since $S \to T$ is a $\bbP^r$-bundle we have
(thanks to Orlov \cite{orlovprojbund}) the following semi-orthogonal decomposition 
$\perf(S) = \langle \perf(T),\ldots, \perf(T)\rangle$ with $r+1$ factors. Consequently, since by hypothesis ${\bf J}(\perf^\dg(T))=0$, all the conditions of Theorem~\ref{thm:main1} are verified. This achieves the proof.
\end{proof}
\subsection*{Proof of Theorem~\ref{thm:hpd-gentype}}
\begin{proof}
Similarly to the proof of Theorem~\ref{thm:morita-hyperbolic},
it suffices to verify all conditions of item (ii) of Theorem~\ref{thm:main1}.
Note first that the composition $ \Phi:\perf(X) \simeq \perf(S,\cC_0) \hookrightarrow \perf(Q)$ is a fully faithful functor between derived categories of smooth projective $\bbC$-schemes.
Consequently, it is of Fourier-Mukai type; see Orlov \cite{orlov-represent}. 
Following the notations of Theorem~\ref{thm:main1}, we can then take $\cT_X:= \perf(X)$. Since $\cT_X^\perp=0$ we have ${\bf J}(\cT_X^{\perp,\dg})=0$. Recall that by assumption the bilinear pairings \eqref{eq:pairings1} (associated to $X$ and $Q$) are non-degenerate. Consequently, all the conditions of Theorem~\ref{thm:main1} are verified and so the proof is finished.
\end{proof}


\begin{thebibliography}{00}


\bibitem{alekseev:dP4}
V. A.\ Alekseev, \emph{On conditions for the rationality of three-folds with a
  pencil of del {P}ezzo surfaces of degree {$4$}}. Mat. Zametki \textbf{41}
  (1987), no.~5, 724--730, 766.

\bibitem{Andre} Y.~Andr{\'e}, {\em Une introduction aux motifs (motifs purs, motifs mixtes, p{\'e}riodes)}. Panoramas et Synth{\`e}ses {\bf 17}. Soci{\'e}t{\'e} Math{\'e}matique de France, Paris, 2004.


\bibitem{auel-bernar-bologn} A.~Auel, M.~Bernardara and M.~Bolognesi, \emph{Fibrations in complete intersections of quadrics, Clifford algebras,
derived categories, and rationality problems}. To appear on Journal de Math. Pures et appliqu\'ees, preprint available at arXiv:1109.6938.

\bibitem{beauvilleprym}
A.~Beauville, \emph{Vari\'et\'es de {P}rym et jacobiennes interm\'ediaires}. Ann. Sci. de l'ENS \textbf{10} (1977), 309--391.


\bibitem{beltrachow}
M.~Beltrametti, \emph{On the {C}how group and the intermediate jacobian of a
  conic bundle}. Annali di Mat. pura e applicata \textbf{141} (1985), no.~4,
  331--351.

  
\bibitem{berna_macri_mehro_stella}
M.\ Bernardara, E.\ Macr{\`\i}, S.\ Mehrotra, and P.\ Stellari, \emph{A categorical invariant for cubic threefolds}. 
Advances in Math. {\bf 229} (2), 770-803 (2012)


\bibitem{bolognesi_bernardara:conic_bundles}
M.\ Bernardara and M.\ Bolognesi, \emph{Derived categories and
  rationality of conic bundles}, Compositio Math. {\bf 149} (11), 1789-1817 (2013).

\bibitem{bolognesi_bernardara:representability}
\bysame, \emph{Categorical representability
  and intermediate Jacobians of Fano threefolds}. EMS Ser. Congr. Rep., Eur. Math. Soc. (2013).

\bibitem{marcello-goncalo-chowgroups}
M.\ Bernardara and G.\ Tabuada, \emph{Chow groups of intersections of quadrics via homological projective duality and 
(Jacobians of) noncommutative motives}. Available at arXiv:1310.6020.

\bibitem{birlange} C. Birkenhake and H. Lange, {\it Complex Abelian Varieties}. Grundlehren 
der Math. Wissenschaften {\bf 302} (1992), Springer-Verlag.

\bibitem{blochmurreFano}
S.~Bloch and J.~P. Murre, \emph{On the {C}how group of certain types of {F}ano
  threefolds}. Compositio Math. \textbf{39} (1979), no.~1, 47--105.

\bibitem{bondal-orlov}
A.~Bondal and D.~Orlov, {\em Semiorthogonal decomposition for algebraic varieties}. Available at arXiv math.AG/9506012.

\bibitem{bondal_orlov:ICM2002}
\bysame,
\emph{Derived categories of coherent sheaves},
Proceedings of the International Congress of Mathematicians, Vol. II
(Beijing, 2002), 47--56, Higher Ed.\ Press, Beijing, 2002. 
  

\bibitem{bondorl-reconst}
\bysame, {\em Reconstruction of a variety from the derived category and groups of autoequivalences}, Compositio Math. 125 (2001), no. 3, 327--344.

\bibitem{bridge-flop}
T. Bridgeland, {\em Flops and derived categories}. Invent. Math. {\bf 147} (2002), 613--632.

\bibitem{Caldararu} A.~Caldararu and S.~Willerton, {\em The Mukai pairing I: a categorical approach}. New York J. math {\bf 16} (2010), 61-98.


\bibitem{ceresaverra}
G.~Ceresa and A.~Verra, \emph{The {A}bel-{J}acobi isomorphism for the sextic
  double solid}. Pacific J. Math. \textbf{124} (1986), no.~1, 85--105.

\bibitem{clemens4ic} C. H. Clemens, \emph{Double solids}. Advances in Math. {\bf 47} (1983), 107--230.

\bibitem{clemensgriffiths}
C. Clemens and P.~Griffiths, \emph{The intermediate {J}acobian of the cubic
  threefold}. Ann. Math. \textbf{95} (1972), 281--356.

\bibitem{debarre-qci} O. Debarre, {\em Le th\'eor\`eme de Torelli pour les intersections de trois quadriques}. Inv. Math. {\bf 95},
(1989), 507--528.


\bibitem{debarre-quartic-double} \bysame,
{\em Sur le th\'eor\`eme de Torelli pour les solides doubles quartiques}. Compositio Math. {\bf 73} (1990), no. 2, 161--187. 

\bibitem{debarre-iliev-manivel}
O.~Debarre, A.~Iliev and L.~Manivel, \emph{On the period map for prime {F}ano
threefolds of degree 10}. J. Algebraic Geometry, {\bf 21} (2012) 21--59.

\bibitem{donagi-torelli}
R.~Donagi, {\em Group law on the intersection of two quadrics}. Annali della Scuola Normale Superiore di Pisa - Classe di Scienze,
Ser. 4, 7 no. 2 (1980), p. 217--239. 

\bibitem{Griffiths} P.~Griffiths, {\em On the periods of certain rational integrals I, II}. Ann. of Math. (2) {\bf 90} (1969), 460--495.

\bibitem{gorch-gul-motives-and-repr} S.~Gorchinskiy and V.~Guletskii, \emph{Motives and representability of algebraic cycles on threefolds over a field} J. Algebraic Geom. {\bf 21} (2012), no. 2, 347--373. 

\bibitem{hartshorne} R. Hartshorne, \emph{Algebraic Geometry}. Graduate Texts in Mathematics {\bf 52}. Springer-Verlag, New York-Heidelberg, 1977.

\bibitem{huybrechts}
D. Huybrechts,
{\em Fourier-{M}ukai transforms in {A}lgebraic {G}eometry}.
Oxford Math. Monongraphs (2006).

\bibitem{iliev10}
A.~Iliev, \emph{The {F}ano surface of the {G}ushel threefold},
  Compositio Math. \textbf{94} (1994), no.~1, 81--107.

\bibitem{ilievV16}
\bysame, \emph{The {${\rm Sp}_3$}-{G}rassmannian and duality for prime {F}ano
  threefolds of genus 9}. Manuscripta Math. \textbf{112} (2003), no.~1, 29--53.

\bibitem{ilievamanovella}
A.~Iliev and L.~Manivel, \emph{Prime {F}ano threefolds and integrable systems},
  Math. Ann. \textbf{339} (2007), no.~4, 937--955.

\bibitem{iliev-marku}
A. Iliev, and D. Markushevich,
{\em The Abel-Jacobi map for cubic threefold and periods of Fano threefolds of degree 14}. Doc. Math. J. {\bf 5}, (2000), 23--47.

\bibitem{ilievmarkuV12}
\bysame, \emph{Elliptic curves and rank-2 vector bundles on the prime {F}ano
  threefold of genus 7}. Adv. Geom. \textbf{4} (2004), no.~3, 287--318.

\bibitem{isko-prok-fano}
V.~A. Iskovskikh and Yu.~G. Prokhorov, \emph{Fano varieties}. Algebraic geometry V. Volume {\bf 47} of Encyclopaedia Math. Sci., Springer, Berlin, 1999.

\bibitem{kanevdp1}
V.~I. Kanev, \emph{Intermediate {J}acobians of threefolds with a pencil of
  {D}el {P}ezzo surfaces and generalized {P}rym varieties}. C. R. Acad. Bulgare
  Sci. \textbf{36} (1983), no.~8, 1015--1017.

\bibitem{kanevdp2}
\bysame, \emph{Intermediate {J}acobians and {C}how groups of three-folds with a
  pencil of del {P}ezzo surfaces}. Ann. Mat. Pura Appl. (4) \textbf{154}
  (1989), 13--48.

\bibitem{ICM-Keller} B.~Keller, {\em On differential graded
    categories}. International Congress of Mathematicians (Madrid), Vol.~II,
  151--190, Eur.~Math.~Soc., Z{\"u}rich, 2006.

\bibitem{kollar-book}
J. Koll\`ar, {\em Rational Curves on Algebraic Varieties}. Ergebnisse der Math. {\bf 32}, Springer-Verlag,
Berlin, 1996.


\bibitem{IAS} M.~Kontsevich, {\em Noncommutative motives}. Talk at the Institute for Advanced Study on the occasion of the $61^{\mathrm{st}}$ birthday of Pierre Deligne, October 2005. Video available at \url{http://video.ias.edu/Geometry-and-Arithmetic}.    
    
\bibitem{Miami} \bysame, {\em Mixed noncommutative motives}. Talk at the Workshop on Homological Mirror Symmetry,
Miami (2010). Notes available at \url{www-math.mit.edu/auroux/frg/miami10-notes}.

\bibitem{finMot} \bysame, {\em Notes on motives in finite characteristic}.  Algebra, arithmetic, and geometry: in honor of Yu. I. Manin. Vol. II,  213--247, Progr. Math., {\bf 270}, BirkhŠuser Boston, MA, 2009.    

\bibitem{kuznet-v14}
A. Kuznetsov,
{\em Derived categories of cubic and $V_{14}$ threefolds}. Proc. Steklov Inst. Math. {\bf 246}, 171--194 (2004); translation from Tr. Mat. Inst. Steklova {\bf 246}, 183--207 (2004).

\bibitem{kuznet-grass}
\bysame, \emph{Homological projective duality for Grassmannians of lines}. Available at arXiv:0610957.

\bibitem{kuznetsov:hpd}
\bysame, \emph{Homological projective duality}. Publ. IH{\'E}S {\bf 105}\ (2007), 157--220.

\bibitem{kuznetfanothreefolds}
\bysame, \emph{Derived categories of {F}ano threefolds}. Proc. Steklov Inst. Math. {\bf 264} (2009), no. 1, 110--122.

\bibitem{kuznetbasechange} \bysame, \emph{Base change for semiorthogonal decompositions}.
Comp. Math. {\bf 147} (2011), no.~3,
852--876.

\bibitem{kuznetquadrics} \bysame, {\em Derived categories of quadric fibrations and intersections of quadrics}. Adv. in Math. {\bf 218} (2008), no.~5, 1340--1369.

\bibitem{laszlo} Y. Laszlo, {\em Th\'eor\`eme de Torelli g\'en\'erique pour les intersections compl\`etes de trois quadriques
de dimension paire}. Invent. Math. {\bf 98} (1989), no. 2, 247--264. 

\bibitem{logachevV10}
D.~Yu. Logach{\"e}v, \emph{Isogeny of the {A}bel-{J}acobi mapping for a {F}ano threefold of genus six}. Available at arXiv:0407147.
  
\bibitem{LO} V.~Lunts and D.~Orlov, {\em Uniqueness of enhancement for triangulated categories}. J. Amer. Math. Soc. {\bf 23} (2010), no.~3, 853--908. 

\bibitem{MT} M.~Marcolli and G.~Tabuada, {\em Jacobians of noncommutative motives}. Moscow Mathematical Journal
{\bf 14} (2014), no. 3, 577--594.

\bibitem{Semi} \bysame, {\em Noncommutative motives, numerical equivalence, and semi-simplicity}. American Journal of
Mathematics {\bf 136} (2014), no. 1, 59--75.

\bibitem{merindol} J.~Y. M\'erindol,
{\em Th\'eor\`eme de Torelli affine pour les intersections de deux quadriques}, Invent. Math. {\bf 80} (1985), no. 3, 375--416. 

\bibitem{mukai} S. Mukai, {\em Duality between $D(X)$ and $D(\widehat{X})$ and its application to Picard sheaves}. Nagoya Math. J., {\bf 81} 153--175, 1981.

\bibitem{mukaigranmisto}
\bysame, \emph{Noncommutativizability of {B}rill-{N}oether theory and
  {$3$}-dimensional {F}ano varieties}. S\=ugaku \textbf{49} (1997), no.~1,
  1--24.

\bibitem{mumford-book} D.~Mumford, {\em Abelian varieties}. Tata Institute of Fundamental Research Studies in Mathematics (1970).

\bibitem{orlovprojbund}
D.~O. Orlov, \emph{Projective bundles, monoidal transformations and derived
  categories of coherent sheaves}. Russian Math. Izv. \textbf{41} (1993),
  133--141.

\bibitem{orlov-represent}
\bysame,
{\em Derived categories of coherent sheaves and equivalences between them}.
Russian Math. Surveys {\bf 58} (2003), 511--591.

\bibitem{Pappas} G.~Pappas, {\em Integral Grothendieck-Riemann-Roch theorem}. Invent. Math. {\bf 170} (2007), no.~3, 455--481. 

\bibitem{polishchuk}
A.~Polishchuk, {\em Holomorphic bundles on 2-dimensional noncommutative toric orbifolds},
Noncommutative geometry and number theory, 341--359, Aspects Math., E37, Vieweg, Wiesbaden, 2006. 

\bibitem{reid:thesis}
M.\ Reid, \emph{The complete intersection of two or more quadrics}. Ph.D.\
  thesis, Trinity College, Cambridge. Available at \url{http://homepages.warwick.ac.uk/~masda/3folds/qu.pdf}

\bibitem{IMRN} G.~Tabuada, {\em Additive invariants of dg categories}. Int. Math. Res. Not. {\bf 53} (2005), 3309--3339.   

\bibitem{survey} \bysame, {\em A guided tour through the garden of noncommutative motives}. Clay Mathematics Proceedings, Volume {\bf 16} (2012), 259--276.

\bibitem{CvsNC} \bysame, {\em Chow motives versus noncommutative motives}. Journal of Noncommutative Geometry {\bf 7} (2013), no.~3, 767--786.

\bibitem{tihoquarticsolid}
A.~S. Tihomirov, \emph{The intermediate {J}acobian of double {${\bf P}^{3}$}
  that is branched in a quartic}. Izv. Akad. Nauk SSSR Ser. Mat. \textbf{44}
  (1980), no.~6, 1329--1377, 1439.

\bibitem{Toen} B.~To{\"e}n, {\em The homotopy theory of dg-categories and
    derived Morita theory}. Invent. Math. {\bf 167} (2007), no.~3, 615--667.

\bibitem{torelli} R.~Torelli, {\em Sulle variet\`a di Jacobi}. Atti Accad, Lincei Rend. C1. Sc. fis.
mat. nat. {\bf 22} (1913), 98--103.

\bibitem{Vial} C.~Vial, {\em Projectors on the intermediate algebraic Jacobians}. New York J. Math. {\bf 19} (2013) 793--822. 

\bibitem{vial-fibrations} \bysame, {\em Algebraic cycles and fibrations}. Documenta Math. {\bf 18} (2013) 1521--1553.

\bibitem{voisin-quartic-double} C.~Voisin, {\em Sur la jacobienne interm\'ediaire du double solide d'indice deux}. Duke Math. J. {\bf 57},
(1988), 629--646.


\end{thebibliography}
\end{document}